\documentclass[10pt]{article}
\usepackage{amsmath}
\usepackage{amssymb,latexsym,xcolor,dsfont,enumitem,mathrsfs,stmaryrd,bm,mathtools} %,cite}
\usepackage{cite}
\usepackage[colorlinks=true,linkcolor=blue!50!black,pagecolor=blue!50!black,citecolor=green!50!black]{hyperref}
\usepackage[auth-sc]{authblk}
\usepackage{multirow}
\usepackage{longtable}
\usepackage[left=2.1cm,right=2.1cm, top=2.5cm,bottom=2.5cm,bindingoffset=0cm]{geometry}

\usepackage{tikz}
\usepackage{pgf}
\usetikzlibrary{calc}
\usetikzlibrary{intersections}
\usetikzlibrary{fadings}

\usepackage[cp1251]{inputenc}
\usepackage[T1]{fontenc}
\usepackage[russian,english]{babel}

\newlength{\templength}
\newlength{\templengtha}
\newlength{\templengthb}
\newlength{\templengthc}
\newlength{\templengthd}

\newlength{\templengths}
\newlength{\templengthw}

\definecolor{mathcolor}{RGB}{0,75,105} %{RGB}{0,0,0}%{RGB}{0,75,105} % чтобы выделить "свою" математику
\definecolor{argcolor}{RGB}{0,105,50} % {RGB}{0,0,0}%{RGB}{0,105,50} % чтобы выделить аргументы в "своей" математике
\definecolor{speccolor}{RGB}{0,40,125} %{RGB}{0,0,0}%{RGB}{0,40,125} % чтобы выделить "свой" текст
\newcommand{\dismath} [1] {{\color{mathcolor}#1}}
\newcommand{\distext} [1] {{\color{speccolor}#1}}
\newcommand{\argument}[1] {{\color{argcolor!50!black}{#1}}}

\SetEnumitemKey{midsep}{itemsep=0pt}

\setlist{midsep}

\newcommand{\defnotion}       [1]       {\distext{\textit{#1\/}}}

\newcommand{\remember}        [1]       {}
  \newcommand{\Rem}           [1]       {\remember{#1}}

\newcommand{\otuple}          [1]       {\dismath{\langle\argument{#1}\rangle}}

\newcommand{\clmodels}        [0]   %{\modmodels}
                                    {\mathrel{\dismath{|\hspace{-1.25pt}{\models}}}}
\newcommand{\nameKFrame}      [1]       {\dismath{\mathfrak{#1}}}       % имя шкалы Крипке
\newcommand{\nameKModel}      [1]       {\dismath{\mathfrak{#1}}}       % имя модели Крипке
\newcommand{\kframe}          [1]       {\nameKFrame{#1}}       % имя шкалы Крипке
\newcommand{\kmodel}          [1]       {\nameKModel{#1}}       % имя модели Крипке
\newcommand{\kFrame}          [1]       {\dismath{\bfrak{#1}}}       % имя шкалы Крипке
\newcommand{\kModel}          [1]       {\dismath{\bfrak{#1}}}       % имя модели Крипке
\newcommand{\logic}           [1]       {\dismath{\mathbf{#1}}}         % имя логики
\newcommand{\lang}            [1]       {\dismath{\mathcal{#1}}}        % имя языка
\newcommand{\sclass}          [1]       {\dismath{\mathscr{#1}}}        % имя класса структур
  \newcommand{\scls}[1]{\sclass{#1}}
  \newcommand{\Scls}[1]{\Sclass{#1}}

\newcommand{\numbers}         [1]       {\dismath{\mathds{#1}}}
\newcommand{\numN}                      {\numbers{N}}
\newcommand{\numNp}                     {\numbers{N}^+}

\newcommand{\numZ}                      {\numbers{Z}}
\newcommand{\numQ}                      {\numbers{Q}}
\newcommand{\numR}                      {\numbers{R}}

\newcommand{\Pow}             [1]       {\dismath{\mathscr{P}(\argument{#1})}}  % множество подмножеств

\newcommand{\boldfrak}        [1]       {\boldsymbol{\frak{#1}}}

\newcommand{\bfrak}           [1]       {\boldfrak{#1}}

\newcommand{\implication}               {\to}
  \newcommand{\imp} {\implication}
\newcommand{\conjunction}             {\wedge}   % команда имеется в wasysym
  \newcommand{\con} {\conjunction}
\newcommand{\disjunction}               {\vee}
  \newcommand{\dis} {\disjunction}

\newcommand{\equivalence}               {\leftrightarrow}
  \newcommand{\lra} {\equivalence}

\newcommand{\leftsquare}[1]{%
\begin{tikzpicture}[scale=#1]
\draw (0,0)--(0,1)--(1,1)--(1,0)--cycle;
\draw [fill, color=black!25] (0,0)--(0.5,0.5)--(0,1);
\draw (0,0)--(0.5,0.5)--(0,1);
\draw (0,1)--(0,0);
\end{tikzpicture}
}
\newcommand{\rightsquare}[1]{%
\begin{tikzpicture}[scale=#1]
\draw (0,0)--(0,1)--(1,1)--(1,0)--cycle;
\draw [fill, color=black!25] (1,0)--(0.5,0.5)--(1,1);
\draw (1,0)--(0.5,0.5)--(1,1);
\draw (1,1)--(1,0);
\end{tikzpicture}
}

\newcommand{\upsquare}[1]{%
\begin{tikzpicture}[scale=#1]
\draw (0,0)--(0,1)--(1,1)--(1,0)--cycle;
\draw [fill, color=black!25] (0,1)--(0.5,0.5)--(1,1);
\draw (0,1)--(0.5,0.5)--(1,1);
\draw (1,1)--(0,1);
\end{tikzpicture}
}

\newcommand{\downsquare}[1]{%
\begin{tikzpicture}[scale=#1]
\draw (0,0)--(0,1)--(1,1)--(1,0)--cycle;
\draw [fill, color=black!25] (0,0)--(0.5,0.5)--(1,0);
\draw (0,0)--(0.5,0.5)--(1,0);
\draw (1,0)--(0,0);
\end{tikzpicture}
}

\newcommand{\leftsq} {\mathop{\leftsquare {0.180}}}
\newcommand{\rightsq}{\mathop{\rightsquare{0.180}}}
\newcommand{\upsq}   {\mathop{\upsquare   {0.180}}}
\newcommand{\downsq} {\mathop{\downsquare {0.180}}}

\renewcommand{\iff}                     {\mathrel{\dismath{\Longleftrightarrow}}} %{\Leftarrow\!\!\!\Rightarrow}
\newcommand{\imply}                     {\mathrel{\dismath{\Longrightarrow}}}     %{=\!\!\!\Rightarrow}
\newcommand{\bydef}                     {\mathrel{\dismath{\leftrightharpoons}}}  %{=\!\!\!\Rightarrow}

\makeatletter
\newcommand{\mref}{\@ifnextchar({\mref@i}{\mref@i({\Box},{p})}}
\def\mref@i(#1,#2){#1#2 \implication #2}
\makeatother
    \makeatletter
    \newcommand{\mrefp}{\@ifnextchar({\mrefp@i}{\mrefp@i({p})}}
    \def\mrefp@i(#1){\mref(\Box,#1)}
    \makeatother
\makeatletter
\newcommand{\FOref}{\@ifnextchar({\FOref@i}{\FOref@i({x},{P})}}
\def\FOref@i(#1,#2){\forall #1\,#2(#1,#1)}
\makeatother
    \makeatletter
    \newcommand{\FOrefp}{\@ifnextchar({\FOrefp@i}{\FOrefp@i({P})}}
    \def\FOrefp@i(#1){\FOref(x,#1)}
    \makeatother
\makeatletter
\newcommand{\FOrefi}{\@ifnextchar({\FOrefi@i}{\FOrefi@i({x},{P})}}
\def\FOrefi@i(#1,#2){\forall #1\,#1#2#1}
\makeatother
    \makeatletter
    \newcommand{\FOrefip}{\@ifnextchar({\FOrefip@i}{\FOrefip@i({P})}}
    \def\FOrefip@i(#1){\FOrefi(x,#1)}
    \makeatother

\makeatletter
\newcommand{\mtra}{\@ifnextchar({\mtra@i}{\mtra@i({\Box},{p})}}
\def\mtra@i(#1,#2){#1#2 \implication #1#1#2}
\makeatother
    \makeatletter
    \newcommand{\mtrap}{\@ifnextchar({\mtrap@i}{\mtrap@i({p})}}
    \def\mtrap@i(#1){\mtra(\Box,#1)}
    \makeatother
\makeatletter
\newcommand{\FOtra}{\@ifnextchar({\FOtra@i}{\FOtra@i({x},{y},{z},{P})}}
\def\FOtra@i(#1,#2,#3,#4){\forall #1\forall #2\forall #3\,(#4(#1,#2)\conjunction #4(#2,#3) \implication #4(#1,#3))}
\makeatother
    \makeatletter
    \newcommand{\FOtrap}{\@ifnextchar({\FOtrap@i}{\FOtrap@i({P})}}
    \def\FOtrap@i(#1){\FOtra(x,y,z,#1)}
    \makeatother
\makeatletter
\newcommand{\FOtrai}{\@ifnextchar({\FOtrai@i}{\FOtrai@i({x},{y},{z},{P})}}
\def\FOtrai@i(#1,#2,#3,#4){\forall #1\forall #2\forall #3\,(#1#4#2\conjunction #2#4#3 \implication #1#4#3)}
\makeatother
    \makeatletter
    \newcommand{\FOtraip}{\@ifnextchar({\FOtraip@i}{\FOtraip@i({P})}}
    \def\FOtraip@i(#1){\FOtrai(x,y,z,#1)}
    \makeatother

\makeatletter
\newcommand{\msym}{\@ifnextchar({\msym@i}{\msym@i({\Box},{\Diamond},{p})}}
\def\msym@i(#1,#2,#3){#3 \implication #1#2#3}
\makeatother
    \makeatletter
    \newcommand{\msymp}{\@ifnextchar({\msymp@i}{\msymp@i({p})}}
    \def\msymp@i(#1){\msym(\Box,\Diamond,#1)}
    \makeatother
\makeatletter
\newcommand{\FOsym}{\@ifnextchar({\FOsym@i}{\FOsym@i({x},{y},{P})}}
\def\FOsym@i(#1,#2,#3){\forall #1\forall #2\,(#3(#1,#2)\implication #3(#2,#1))}
\makeatother
    \makeatletter
    \newcommand{\FOsymp}{\@ifnextchar({\FOsymp@i}{\FOsymp@i({P})}}
    \def\FOsymp@i(#1){\FOsym(x,y,#1)}
    \makeatother
\makeatletter
\newcommand{\FOsymi}{\@ifnextchar({\FOsymi@i}{\FOsymi@i({x},{y},{P})}}
\def\FOsymi@i(#1,#2,#3){\forall #1\forall #2\,(#1#3#2 \implication #2#3#1)}
\makeatother
    \makeatletter
    \newcommand{\FOsymip}{\@ifnextchar({\FOsymip@i}{\FOsymip@i({P})}}
    \def\FOsymip@i(#1){\FOsymi(x,y,#1)}
    \makeatother

\makeatletter
\newcommand{\meuc}{\@ifnextchar({\meuc@i}{\meuc@i({\Box},{\Diamond},{p})}}
\def\meuc@i(#1,#2,#3){#2#3 \implication #1#2#3}
\makeatother
    \makeatletter
    \newcommand{\meucp}{\@ifnextchar({\meucp@i}{\meucp@i({p})}}
    \def\meucp@i(#1){\meuc(\Box,\Diamond,#1)}
    \makeatother
\makeatletter
\newcommand{\FOeuc}{\@ifnextchar({\FOeuc@i}{\FOeuc@i({x},{y},{z},{P})}}
\def\FOeuc@i(#1,#2,#3,#4){\forall #1\forall #2\forall #3\,(#4(#1,#2)\con #4(#1,#3)\implication #4(#2,#3))}
\makeatother
    \makeatletter
    \newcommand{\FOeucp}{\@ifnextchar({\FOeucp@i}{\FOeucp@i({P})}}
    \def\FOeucp@i(#1){\FOeuc(x,y,z,#1)}
    \makeatother
\makeatletter
\newcommand{\FOeuci}{\@ifnextchar({\FOeuci@i}{\FOeuci@i({x},{y},{z},{P})}}
\def\FOeuci@i(#1,#2,#3,#4){\forall #1\forall #2\forall #3\,(#1#4#2 \con #1#4#3 \implication #2#4#3)}
\makeatother
    \makeatletter
    \newcommand{\FOeucip}{\@ifnextchar({\FOeucip@i}{\FOeucip@i({P})}}
    \def\FOeucip@i(#1){\FOeuci(x,y,z,#1)}
    \makeatother

\makeatletter
\newcommand{\mser}{\@ifnextchar({\mser@i}{\mser@i({\Box},{\Diamond},{p})}}
\def\mser@i(#1,#2,#3){#1#3 \implication #2#3}
\makeatother
    \makeatletter
    \newcommand{\mserp}{\@ifnextchar({\mserp@i}{\mserp@i({p})}}
    \def\mserp@i(#1){\mser(\Box,\Diamond,#1)}
    \makeatother
\makeatletter
\newcommand{\FOser}{\@ifnextchar({\FOser@i}{\FOser@i({x},{y},{P})}}
\def\FOser@i(#1,#2,#3){\forall #1\exists #2\,#3(#1,#2)}
\makeatother
    \makeatletter
    \newcommand{\FOserp}{\@ifnextchar({\FOserp@i}{\FOserp@i({P})}}
    \def\FOserp@i(#1){\FOser(x,y,#1)}
    \makeatother
\makeatletter
\newcommand{\FOseri}{\@ifnextchar({\FOseri@i}{\FOseri@i({x},{y},{P})}}
\def\FOseri@i(#1,#2,#3){\forall #1\exists #2\,#1#3#2}
\makeatother
    \makeatletter
    \newcommand{\FOserip}{\@ifnextchar({\FOserip@i}{\FOserip@i({P})}}
    \def\FOserip@i(#1){\FOseri(x,y,#1)}
    \makeatother

\makeatletter
\newcommand{\mla}{\@ifnextchar({\mla@i}{\mla@i({\Box},{p})}}
\def\mla@i(#1,#2){#1(#1#2 \implication #2) \implication #1#2}
\makeatother
    \makeatletter
    \newcommand{\mlap}{\@ifnextchar({\mlap@i}{\mlap@i({p})}}
    \def\mlap@i(#1){\mla(\Box,#1)}
    \makeatother

\makeatletter
\newcommand{\mgrz}{\@ifnextchar({\mgrz@i}{\mgrz@i({\Box},{p})}}
\def\mgrz@i(#1,#2){#1(#1(#2 \implication #1#2) \implication #2) \implication #2}
\makeatother
    \makeatletter
    \newcommand{\mgrzp}{\@ifnextchar({\mgrzp@i}{\mgrzp@i({p})}}
    \def\mgrzp@i(#1){\mgrz(\Box,#1)}
    \makeatother

\makeatletter
\newcommand{\mwgrz}{\@ifnextchar({\mwgrz@i}{\mwgrz@i({\Box},{p})}}
\def\mwgrz@i(#1,#2){#1^+(#1(#2 \implication #1#2) \implication #2) \implication #2}
\makeatother
    \makeatletter
    \newcommand{\mwgrzp}{\@ifnextchar({\mwgrzp@i}{\mwgrzp@i({p})}}
    \def\mwgrzp@i(#1){\mwgrz(\Box,#1)}
    \makeatother

\makeatletter
\newcommand{\ExtDiamond}{\@ifnextchar({\ExtDiamond@i}{\ExtDiamond@i({p})}}
\def\ExtDiamond@i(#1){{\Diamond\!\!\!\!\Diamond}_{#1}}
\makeatother

\renewcommand{\dismath} [1] {#1}
\renewcommand{\distext} [1] {#1}
\renewcommand{\argument}[1] {#1}

\newcounter{\theequation}[section]
\renewcommand{\theequation}{\thesection.\arabic{equation}}

\renewcommand{\kModel} [1]{\kmodel{#1}}
\renewcommand{\Scls}   [1]{\scls{#1}}

\newcommand{\QSIL}     [0]{\mathop{\mbox{\textnormal{\texttt{QSIL}}}}}
\newcommand{\QSILext}  [2]{\mathop{\mbox{\textnormal{\texttt{QSIL}}}^{\mathit{#1}}_{\mathit{#2}}}}

\newcommand{\aug}      [2]{\mathop{\mbox{\textnormal{\texttt{aug}}}^{\mathit{#1}}_{\mathit{#2}}}}

\makeatletter
\def\thmstyle{\it} % style of text in theorem environment
\def\@begintheorem#1#2{\it \trivlist \item[\hskip
        \labelsep{\bf #1\ #2.}]\thmstyle}
\def\@opargbegintheorem#1#2#3{\it \trivlist \item[\hskip
        \labelsep{\bf #1\ #2\ (#3).}]\thmstyle}
\makeatother

\newtheorem{theorem}{{\indent}Theorem}[section]

\newtheorem{lemma}[theorem]{{\indent}Lemma}
\newtheorem{proposition}[theorem]{{\indent}Proposition}
\newtheorem{sublemma}[theorem]{{\indent}Sublemma}

\newtheorem{corollary}[theorem]{{\indent}Corollary}

\newcommand{\numeral}[1]{\overline{#1}} %{\ulcorner\!#1\!\urcorner}

\definecolor{cg0}{gray}{1.00}          % It's mine, use it here
\definecolor{cg1}{gray}{0.90}          % It's mine, use it here
\definecolor{cg2}{gray}{0.80}          % It's mine, use it here
\definecolor{cg3}{gray}{0.70}          % It's mine, use it here
\definecolor{cg4}{gray}{0.60}          % It's mine, use it here
\definecolor{cg5}{gray}{0.50}          % It's mine, use it here
\definecolor{cg6}{gray}{0.40}          % It's mine, use it here
\definecolor{cg7}{gray}{0.30}          % It's mine, use it here
\definecolor{cg8}{gray}{0.20}          % It's mine, use it here
\definecolor{cg9}{gray}{0.10}          % It's mine, use it here
 
\newcommand{\nicearrow}[3]
{
}

\newcommand{\drawtileflattm} [9]
{
\draw [white, opacity = 0, name path = diag 1] #1--#3;
\draw [white, opacity = 0, name path = diag 2] #2--#4;
\draw [name intersections = {of = diag 1 and diag 2, by = {tcenter}}];

\foreach \c in {0.96}
{
  \coordinate (ttl)     at ($(tcenter)+\c*#4-\c*(tcenter)$);
  \coordinate (tbr)     at ($(tcenter)+\c*#2-\c*(tcenter)$);
  \coordinate (tbl)     at ($(tcenter)+\c*#1-\c*(tcenter)$);
  \coordinate (ttr)     at ($(tcenter)+\c*#3-\c*(tcenter)$);
  \coordinate (sttl)    at ($(tcenter)+0.4*(ttl)-0.4*(tcenter)$);
  \coordinate (stbr)    at ($(tcenter)+0.4*(tbr)-0.4*(tcenter)$);
  \coordinate (stbl)    at ($(tcenter)+0.4*(tbl)-0.4*(tcenter)$);
  \coordinate (sttr)    at ($(tcenter)+0.4*(ttr)-0.4*(tcenter)$);
}

\coordinate (ledge) at ($0.5*#4+0.5*#1$);
\coordinate (bedge) at ($0.5*#1+0.5*#2$);
\coordinate (redge) at ($0.5*#2+0.5*#3$);
\coordinate (tedge) at ($0.5*#3+0.5*#4$);

\draw (ttl)--(tbl)--(tbr)--(ttr)--cycle;
\draw (sttl)--(stbl)--(stbr)--(sttr)--cycle;
\draw (ttl)--(sttl);
\draw (ttr)--(sttr);
\draw (tbl)--(stbl);
\draw (tbr)--(stbr);

\node [] at (tcenter) {\phantom{$y^l_l$}{#9}\phantom{$y^l_l$}};
\node [fill = white, draw = black!25, rounded corners] at (ledge)   {{$\phantom{iR^i_iy}$}};
\node [fill = white, draw = black!25, rounded corners] at (bedge)   {{$\phantom{iR^i_iy}$}};
\node [fill = white, draw = black!25, rounded corners] at (redge)   {{$\phantom{iR^i_iy}$}};
\node [fill = white, draw = black!25, rounded corners] at (tedge)   {{$\phantom{iR^i_iy}$}};
\node [] at (ledge)   {\phantom{$y^l_l$}{#5}\phantom{$y^l_l$}};
\node [] at (bedge)   {\phantom{$y^l_l$}{#6}\phantom{$y^l_l$}};
\node [] at (redge)   {\phantom{$y^l_l$}{#7}\phantom{$y^l_l$}};
\node [] at (tedge)   {\phantom{$y^l_l$}{#8}\phantom{$y^l_l$}};
}

\newcommand{\drawtileflattms} [4]
{
\draw [white, opacity = 0, name path = diag 1] #1--#3;
\draw [white, opacity = 0, name path = diag 2] #2--#4;
\draw [name intersections = {of = diag 1 and diag 2, by = {tcenter}}];

\foreach \c in {0.96}
{
  \coordinate (ttl)     at ($(tcenter)+\c*#4-\c*(tcenter)$);
  \coordinate (tbr)     at ($(tcenter)+\c*#2-\c*(tcenter)$);
  \coordinate (tbl)     at ($(tcenter)+\c*#1-\c*(tcenter)$);
  \coordinate (ttr)     at ($(tcenter)+\c*#3-\c*(tcenter)$);
  \coordinate (sttl)    at ($(tcenter)+0.4*(ttl)-0.4*(tcenter)$);
  \coordinate (stbr)    at ($(tcenter)+0.4*(tbr)-0.4*(tcenter)$);
  \coordinate (stbl)    at ($(tcenter)+0.4*(tbl)-0.4*(tcenter)$);
  \coordinate (sttr)    at ($(tcenter)+0.4*(ttr)-0.4*(tcenter)$);
}

\coordinate (ledge) at ($0.5*#4+0.5*#1$);
\coordinate (bedge) at ($0.5*#1+0.5*#2$);
\coordinate (redge) at ($0.5*#2+0.5*#3$);
\coordinate (tedge) at ($0.5*#3+0.5*#4$);

\draw (ttl)--(tbl)--(tbr)--(ttr)--cycle;
\draw (sttl)--(stbl)--(stbr)--(sttr)--cycle;
\draw (ttl)--(sttl);
\draw (ttr)--(sttr);
\draw (tbl)--(stbl);
\draw (tbr)--(stbr);

}

\newcommand{\drawtileflattmslanted} [4]
{
\draw [white, opacity = 0, name path = diag 1] #1--#3;
\draw [white, opacity = 0, name path = diag 2] #2--#4;
\draw [name intersections = {of = diag 1 and diag 2, by = {tcenter}}];

\foreach \c in {0.96}
{
  \coordinate (ttl)     at ($(tcenter)+\c*#4-\c*(tcenter)$);
  \coordinate (tbr)     at ($(tcenter)+\c*#2-\c*(tcenter)$);
  \coordinate (tbl)     at ($(tcenter)+\c*#1-\c*(tcenter)$);
  \coordinate (ttr)     at ($(tcenter)+\c*#3-\c*(tcenter)$);
  \coordinate (sttl)    at ($(tcenter)+0.4*(ttl)-0.4*(tcenter)$);
  \coordinate (stbr)    at ($(tcenter)+0.4*(tbr)-0.4*(tcenter)$);
  \coordinate (stbl)    at ($(tcenter)+0.4*(tbl)-0.4*(tcenter)$);
  \coordinate (sttr)    at ($(tcenter)+0.4*(ttr)-0.4*(tcenter)$);
}

\coordinate (ledge) at ($0.5*#4+0.5*#1$);
\coordinate (bedge) at ($0.5*#1+0.5*#2$);
\coordinate (redge) at ($0.5*#2+0.5*#3$);
\coordinate (tedge) at ($0.5*#3+0.5*#4$);

\draw (ttl)--(tbl)--(tbr)--(ttr)--cycle;
\draw (sttl)--(stbl)--(stbr)--(sttr)--cycle;
\draw (ttl)--(sttl);
\draw (ttr)--(sttr);
\draw (tbl)--(stbl);
\draw (tbr)--(stbr);
}

\newcommand{\circleone} [4]
{
\draw [fill = #3] #1 circle [radius = #2];
}

\newcommand{\circletwo} [4]
{
\circleone{#1}{#2}{#3}{#4}
\draw [] #1 circle [radius = #2+#4];
}

\newcommand{\circlethree} [4]
{
\circletwo{#1}{#2}{#3}{#4}
\draw [] #1 circle [radius = #2+2*#4];
}

\newcommand{\circletile} [6]
{
\filldraw[fill=#3, draw=black] #1 -- ($#1+#2*(-0.707,+0.707)$)
                               arc [start angle=135, end angle=225, radius=#2] -- cycle;
\filldraw[fill=#4, draw=black] #1 -- ($#1+#2*(-0.707,-0.707)$)
                               arc [start angle=225, end angle=315, radius=#2] -- cycle;
\filldraw[fill=#5, draw=black] #1 -- ($#1+#2*(+0.707,-0.707)$)
                               arc [start angle=-45, end angle= 45, radius=#2] -- cycle;
\filldraw[fill=#6, draw=black] #1 -- ($#1+#2*(+0.707,+0.707)$)
                               arc [start angle= 45, end angle=135, radius=#2] -- cycle;
}

\newcommand{\drawtileflatsmall} [8]
{
\coordinate (stcenter) at ($0.5*#1+0.5*#3$);

\fill [#5, fill opacity=0.75] #1--(stcenter)--#4--cycle;
\fill [#6, fill opacity=0.75] #2--(stcenter)--#1--cycle;
\fill [#7, fill opacity=0.75] #3--(stcenter)--#2--cycle;
\fill [#8, fill opacity=0.75] #4--(stcenter)--#3--cycle;

\draw #1 -- #3;
\draw #2 -- #4;
\draw #1 -- #2 -- #3 -- #4 -- cycle;
}

\newcommand{\drawhalftileflatsmallh} [6]
{
\coordinate (stcenter) at ($0.5*#1+0.5*#3$);

\fill [#5, fill opacity=0.75] #1--(stcenter)--#4--cycle;
\fill [#6, fill opacity=0.75] #3--(stcenter)--#2--cycle;

\draw #1 -- #3 -- #2 -- #4 -- cycle;
}

\newcommand{\drawhalftileflatsmallv} [6]
{
\coordinate (stcenter) at ($0.5*#1+0.5*#3$);

\fill [#5, fill opacity=0.75] #2--(stcenter)--#1--cycle;
\fill [#6, fill opacity=0.75] #4--(stcenter)--#3--cycle;

\draw #1 -- #3 -- #4 -- #2 -- cycle;
}

\begin{document}

\title{Superintuitionistic predicate logics of linear frames: \\ 
undecidability with two individual variables\thanks{The paper was prepared within the framework of the project ``International academic cooperation'' HSE University.}}
%with three variables
%}
%\tnotetext[t1]{The work is supported by the Basic Research Program of the HSE University.}
%\tnotetext[t1]{The paper was prepared within the framework of the project ``International academic cooperation'' HSE University.}
\author{Mikhail Rybakov}
\affil{{Higher School of Modern Mathematics MIPT}, {HSE University}}
%\email{m\_rybakov@mail.ru}
%\date{\today}

%\begin{keyword}
%     superintuitionistic predicate logic 
%%\sep modal predicate logic 
%\sep two variable fragment
%\sep positive fragment
%\sep decidability
%\MSC[2020] 03B20 \sep 03B25 \sep 03B55 \sep 03D35
%\end{keyword}

\date{}

\maketitle

\begin{abstract}
The paper presents a solution to the long-standing question about the decidability of the two-variable fragment of the superintuitionistic predicate logic~$\logic{QLC}$ defined by the class of linear Kripke frames, which is also the `superintuitionistic' fragment of the modal predicate logic~$\logic{QS4.3}$, under the G\"odel translation. We prove that the fragment is undecidable ($\Sigma^0_1$-complete). The result remains true for the positive fragment, even with a single binary predicate letter and an infinite set of unary predicate letters. 
Also, we prove that the logic defined by ordinal $\omega$ as a Kripke frame is not recursively enumerable (even both $\Sigma^0_1$-hard and $\Pi^0_1$-hard) with the same restrictions on the language. 
The results remain true if we add also the constant domain condition. The proofs are based on two techniques: a modification of the method proposed by M.\,Marx and M.\,Reynolds, which allows us to describe tiling problems using natural numbers rather than pairs of numbers within an enumeration of Cantor's, and an idea of `double labeling' the elements from the domains, which allows us to use only two individual variables in the proof when applying the former method.
\end{abstract}
%~\\[-64pt]

%\newpageafter{abstract}

%\maketitle

%\newpage

%\tableofcontents

%\newpage

%\input{QLC-2var-text-Overleaf}
%%%%%%%%%%%%%%%%%%%%%%%%%%%%%%%%%%%%%%%%%%%%%
\newcommand{\gw}{{\color{green!60!black}w}}
\newcommand{\gf}{{\color{green!60!black}f}}
\newcommand{\gr}{{\color{green!60!black}r}}
\newcommand{\ga}{{\color{green!60!black}a}}
\newcommand{\rw}{{\color{red!60!black}$\bar{\mathrm{w}}$}}
\newcommand{\rf}{{\color{red!60!black}$\bar{\mathrm{f}}$}}
\newcommand{\rr}{{\color{red!60!black}$\bar{\mathrm{r}}$}}
\newcommand{\ra}{{\color{red!60!black}$\bar{\mathrm{a}}$}}

\newcommand{\tile} {{\dismath{\mathit{tile}}}}
\newcommand{\abv}  {{\dismath{\mathit{above}}}}
\newcommand{\rght} {{\dismath{\mathit{right}}}}
\newcommand{\wall} {{\dismath{\mathit{wall}}}}
\newcommand{\floor}{{\dismath{\mathit{floor}}}}
\newcommand{\nxt}  {{\dismath{\mathit{next}}}}
\newcommand{\nxtm} {\nxt^m}
\newcommand{\nmbr} {{\dismath{\textnormal{\texttt{num}}}}}
\newcommand{\rzpair} {{\dismath{\textnormal{\texttt{pair}}}}}
\newcommand{\Abv}  {{\dismath{\textnormal{\texttt{above}}}}}
\newcommand{\Rght} {{\dismath{\textnormal{\texttt{right}}}}}
\newcommand{\Wall} {{\dismath{\textnormal{\texttt{wall}}}}}
\newcommand{\Floor}{{\dismath{\textnormal{\texttt{floor}}}}}
\newcommand{\Nxt}  {{\dismath{\textnormal{\texttt{next}}}}}

\section{Introduction}
\setcounter{equation}{0}

This paper is mainly devoted to solving the open problem of the %algorithmic 
decidability of the superintuitionistic predicate logic $\logic{QLC}$ which can be defined as the logic of linear Kripke frames, or, equivalently, as the logic obtained from the intuitionistic predicate logic $\logic{QInt}$ by adding formulas of the form $(\varphi\to\psi)\vee(\psi\to\varphi)$ as axioms, or, equivalently, as the `superintuitionistic' fragment of the modal predicate logic~$\logic{QS4.3}$; see \mbox{\cite[Section~6.7]{GShS}} for details. This logic is very close to the classical predicate logic: it contains a lot of classical principles rejected by $\logic{QInt}$ and even by $\logic{QKC}$, the logic of the weak low of the excluded middle.

It is well known that the classical predicate logic $\logic{QCl}$ is undecidable~\cite{Church36}, even in the language containing a single binary predicate letter and three individual variables~\cite[Section~4.8~(ii)]{TG87} (see also~\cite{MR:2022:DoklMath,MR:2023:LI}). Then, we readily obtain that $\logic{QLC}$ is undecidable in this language, also. At the same time, some quite expressive fragments of $\logic{QCl}$ are decidable: for example, the monadic fragment, even enriched by equality~\cite[Chapter~21]{BBJ07}, the two-variable fragment~\cite{Mortimer75,GKV97}, and guarded fragments~\cite{Gradel99}; see~\cite{BGG97} for the classical decision problem as well. Nevertheless, non-classical logics are often undecidable even in languages with only monadic predicate letters~\cite{Kripke62,MMO65,MR:2001:LI,MR:2002:LI,MR:2017:LI} or with only two individual variables~\cite{KKZ05} or even both with one-two monadic predicate letters and two individual variables~\cite{RSh19SL,RSh20AiML,RShJLC21c,MR:2024:IGPL}. There are results that provide us with decidable fragments, but they are based on quite strong restrictions on languages or semantics~\cite{HWZ00, HWZ01, WZ01, WZ02, MR:2017:LI, CMRT:2022, ARSh:2023:arXiv, RShsubmitted, RShsubmitted2}. 
%In this paper, some general methods for obtaining results on the undecidability of fragments of classical and non-classical logics and theories will be proposed. With their help, we will get answers to some questions.

Let us return to $\logic{QLC}$. 

On the one hand, the mentioned results show us that non-classical logics can be undecidable in languages with only two individual variables. For example, the two-variable fragment of $\logic{QKC}$ is undecidable even with a single unary predicate letter~\cite{RSh19SL}. Moreover, there are results providing us with the undecidability of modal predicate logics of linear frames. So, the monadic two-variable fragment of modal predicate logic $\logic{QS4.3}$, whose class of Kripke frames is the same as of $\logic{QLC}$, is undecidable~\cite{KKZ05}, even with two unary predicate letters~\cite{RSh:2023:arXiv:lin:2021}.

On the other hand, the methods used in the mentioned papers are not applicable to $\logic{QLC}$. The results for $\logic{QS4.3}$~\cite{KKZ05,RSh:2023:arXiv:lin:2021} are based on the fact that in Kripke semantics the heredity condition is not required for the modal predicate language; but this condition is required in the intuitionistic case. The construction for $\logic{QKC}$~\cite{RSh19SL} essencially uses the fact that Kripke frames validating~$\logic{QKC}$ can contain infinite antichains, that is impossible in Kripke frames validating~$\logic{QLC}$. The techniques presented in the other papers are quite far from being applicable to the two-variable fragment of~$\logic{QLC}$ as well. It is also worth adding that the one-variable fragment of $\logic{QLC}$ is decidable, that immediately follows from the decidability of the propositional bimodal logic~$\logic{S5}\times\logic{S4.3}$ \mbox{\cite[Theorem~6.61]{GKWZ}}; very ineteresting close results can be found in~\cite{CMRT:2022}.

But the methods that allow us to prove the decidability of the mentioned fragments are not applicable to the two-variable fragment of~$\logic{QLC}$, too. As a result, the situation with the decidability of the fragment is so unclear that X.\,Caicedo, G.\,Metcalfe, R.\,Rodr\'{i}guez, and O.\,Tuyt recently noted that it `remains an intriguing open problem'~\cite{CMRT:2022}.\footnote{The author was informed about this issue by Dmitry Shkatov in a private conversation.} 
Here, we fix the situation.

Exactly, we prove that $\logic{QLC}$~--- and even its positive fragment~--- is undecidable in the language with two individual variables, a single binary predicate letter, and an infinite set of unary predicate letters. Then, we also show that the result remains true for some extensions of $\logic{QLC}$, in particular, for $\logic{QLC.cd}$, the logic of linear frames with constant domains. 
The result was presented in~\cite{MR:2024:Piter}.

%The paper is structured as follows.
%We shortly recall the intuitionistic Kripke semantics in Section~\ref{sec:sem}, formulate the undecidable problem we are going to reduce to~$\logic{QLC}$ in Section~\ref{sec:tiling}, and describe a prelimary construction in Section~\ref{sec:prelim}. In Section~\ref{sec:main}, we present the main construction we use, then, in Section~\ref{sec:undec}, infer undecidability of the two-variable fragment of~$\logic{QLC}$ and some its extentions. 

\section{Syntax and semantics}
\label{sec:sem}
\setcounter{equation}{0}

We assume that the intuitionistic predicate language $\lang{L}$ contains countably many individual variables, countably many predicate letters of every arity, the constant $\bot$, the binary connectives $\wedge$, $\vee$, $\to$, the quantifier symbols~$\forall$ and~$\exists$. Formulas in $\lang{L}$, or \defnotion{$\lang{L}$-formulas}, as well as the symbols $\neg$ and $\leftrightarrow$, are defined in the usual way; in particular, $\neg\varphi = \varphi\to \bot$
%, $\top=\neg\bot$, 
and $\varphi \leftrightarrow \psi = (\varphi\to \psi)\wedge (\psi\to\varphi)$. 
%We identify $\lang{L}$ with the set of $\lang{L}$-formulas. 
%For a formula $\varphi$, let $\sub\varphi$ denote the set of subformulas of~$\varphi$. 
%
%A formula $\varphi$ is \defnotion{positive} if $\varphi$ does not contain occurrences of~$\bot$.

A \defnotion{Kripke frame} is a pair $\kframe{F} = \otuple{W,R}$, where $W$ is a non-empty set of \defnotion{possible worlds} and $R$ is a binary \defnotion{accessibility relation} on~$W$. Let, as usual, $R(w) = \{ w' \in W: w R w'\}$; so, $w' \in R(w)$ mean the same as $w R w'$. If $wRw'$ holds, then we say that $w'$ is \defnotion{accessible} from $w$, or that $w$ \defnotion{sees}~$w'$. 
A~Kripke frame $\kframe{F} = \otuple{W,R}$ is called \defnotion{intuitionistic} if $R$ is a partial order~--- i.e., a reflexive, transitive, and antisymmetric binary relation~--- on~$W$. We say that a Kripke frame is \defnotion{linear} if its accessibility relation is a linear order.

We consider only intuitionistic Kripke frames below.

A \defnotion{Kripke frame with domains}, or, for short, an \defnotion{augmented frame}, is a pair $\kFrame{F} = \langle \kframe{F}, D \rangle$, where $\kframe{F} = \otuple{W,R}$ is a Kripke frame and $D$ is a \defnotion{domain function} $D\colon W\to \Pow{\mathcal{D}}$ associating with every $w\in W$ a non-empty subset of a non-empty set $\mathcal{D}$ of \defnotion{individuals}. The set $D(w)$, also denoted by $D_w$, is called \defnotion{the domain of the world\/~$w$}. Sets of the form $D_w$ are also called \defnotion{local domains} of $\kFrame{F}$ and $\mathcal{D}$ is called \defnotion{the global domain} of $\kFrame{F}$; without a loss of generality, we may assume that
$$
\begin{array}{rcl}
\mathcal{D} & = & \displaystyle\bigcup\limits_{\mathclap{w\in W}}D_w.
\end{array}
$$
The augmented frame $\kFrame{F} = \otuple{\kframe{F}, D}$ is also denoted by~$\kframe{F}_D$. We say that $\kframe{F}_D$ is \defnotion{based on $\kframe{F}$}, or is \defnotion{defined over~$\kframe{F}$}. 
%The global domain of $\kframe{F}_D$ is denoted by~$D^+$~\cite{GShS}.

We say that an augmented frame $\kframe{F}_D$ based on a Kripke frame $\kframe{F}=\otuple{W,R}$ satisfies the \defnotion{expanding domain condition} if, for all $u,w \in W$,
%$$
\begin{equation}
\begin{array}{lcl}
  uRw & \Longrightarrow & D_u \subseteq D_{w};
\end{array}
%\eqno{(\mathit{ED})}
\label{eq:ED}
\end{equation}
%$$
then we call $\kframe{F}_D$ \defnotion{augmented frame with expanding domains} or, for short, \defnotion{e\nobreakdash-augmented frame}. We say that  $\kframe{F}_D$ satisfies the \defnotion{locally constant domain condition} if, for all $u,w \in W$,
%$$
\begin{equation}
\begin{array}{lcl}
  uRw & \Longrightarrow & D_u = D_{w},
\end{array}
%\eqno{(\mathit{LCD})}
\label{eq:LCD}
\end{equation}
%$$
and the \defnotion{globally constant domain condition} if, for all $u,w \in W$,
%$$
\begin{equation}
\begin{array}{lcl}
D_u = D_{w}.
\end{array}
%\eqno{(\mathit{GCD})}
\label{eq:GCD}
\end{equation}
%$$
%For our purposes, ${(\mathit{LCD})}$ is sufficient; nevertheless, the frames satisfying ${(\mathit{LCD})}$ we shall consider below, also satisfy ${(\mathit{GCD})}$. 
If $\kframe{F}_D$ satisfies \eqref{eq:LCD}, then we call it \defnotion{augmented frame with constant domains} or, for short, \defnotion{c\nobreakdash-augmented frame}. If $\kframe{F}_D$ satisfies \eqref{eq:GCD} and $\mathcal{D}$ is the global domain of $\kframe{F}_D$, then, following~\cite{GShS}, we also denote it by~$\kframe{F}\odot\mathcal{D}$. For convenience, sometimes we write $\otuple{W,R,D}$ for $\otuple{\kframe{F},D}$ with $\kframe{F}=\otuple{W,R}$. We assume that all augmented frames satisfy~\eqref{eq:ED}, i.e., are e-augmented frames, below; if an e-augmented frame is based on an intuitionistic Kripke frame, then we call it \defnotion{intuitionistic augmented frame}. 

%A \textit{subframe} of an augmented frame $\langle W, R, D \rangle$ is an augmented frame $\langle W', R', D' \rangle$ where $W'$ is a non-empty subset of $W$, and $R' = R \upharpoonright W'$ and $D' = D \upharpoonright W'$.

An \defnotion{intuitionistic predicate Kripke model}, or simply a \defnotion{Kripke model}, is a tuple $\kModel{M} = \langle \kframe{F}_D, I\rangle$, where $\kframe{F}_D = \langle W, R, D \rangle$ is an intuitionistic augmented frame and $I$ is a map, called an \defnotion{interpretation of predicate letters}, assigning to a world $w\in W$ and an $n$-ary predicate letter $P$ an $n$-ary relation $I(w,P)$ on $D_w$ and satisfying the following \defnotion{heredity condition}: 
for all $w, w' \in W$ and every predicate letter~$P$,
$$
%\begin{equation}
\begin{array}{lcl}
  wRw' & \Longrightarrow & I(w,P) \subseteq I(w',P).
\end{array}
%\eqno{(4)}
%\label{eq:4}
%\end{equation}
$$
%; 
We also write 
%$P^{I,w}$ for $I(w,P)$ and 
$\langle W, R, D, I\rangle$ for $\langle \kframe{F}_D, I\rangle$ below. 
%We note that, if a predicate letter $P$ is nullary (i.e., $P$ is a proposition letter), then $P^{I,w} \subseteq D^0_w = \{\otuple{}\}$, for every $w \in W$.  Conceptually, $P^{I,w} = \varnothing$ corresponds to assigning the truth value ``false'', and $P^{I,w} = \{\otuple{}\}$ the truth value ``true'', to~$P$ at~$w$. 
For a Kripke model $\kModel{M} = \langle \kframe{F}_D, I\rangle$, we say that $\kModel{M}$ is \defnotion{based on\/ $\kframe{F}_D$} and is \defnotion{based on\/ $\kframe{F}$}.

An \defnotion{assignment} in a Kripke model $\kModel{M} = \otuple{W, R, D, I}$ is a map $g$ associating with every variable $x$ an element $g(x)$ of the global domain of the augmented frame $\otuple{W, R, D}$. If $g$ and $h$ are assignments such that $g(y) = h(y)$ whenever $y \ne x$, we write $g \stackrel{x}{=} h$.

The truth of an $\lang{L}$-formula $\varphi$ at a world $w$ of an intuitionistic predicate Kripke model $\kModel{M} = \otuple{W,R,D,I}$ under an assignment $g$ is defined recursively:
%%%%%%%%%%%%%%%%%%%%%%%%%%%%%%%%%%%%%%%%%%%%%%%%%%%%%%%%%%%%%%%%%%%%%%
\settowidth{\templength}{\mbox{$\kModel{M},w\models^g\varphi'$ and $\kModel{M},w\models^g\varphi''$;}}
\settowidth{\templengtha}{\mbox{$w$}}
\settowidth{\templengthb}{\mbox{$\kModel{M},w\models^{h}\varphi'$, for every assignment $h$ such that}}
\settowidth{\templengthc}{\mbox{$\kModel{M},w\models^g P(x_1,\ldots,x_n)$}}
\settowidth{\templengthd}{\mbox{$\kModel{M},\parbox{\templengtha}{$v$}\models^{h}\varphi'$, for every $v\in R(w)$ and every assignment;;;}}
%%%%%%%%%%%%%%%%%%%%%%%%%%%%%%%%%%%%%%%%%%%%%%%%%%%%%%%%%%%%%%%%%%%%%%
$$
\begin{array}{lcl}
\kModel{M},w\models^g P(x_1,\ldots,x_n)
  & \leftrightharpoons
  & \parbox{\templengthd}{$\langle g(x_1),\ldots,g(x_n)\rangle \in P^{I, w}$,} \\
\end{array}
$$
\mbox{where $P$ is an $n$-ary predicate letter;}
%%%%%%%%%%%%%%%%%%%%%%%%%%%%%%%%%%%%%%%%%%%%%%%%%%%%%%%%%%%%%%%%%%%%%%
\settowidth{\templength}{\mbox{$\kModel{M},w\models^g\varphi'$ and $\kModel{M},w\models^g\varphi''$;}}
\settowidth{\templengtha}{\mbox{$w$}}
\settowidth{\templengthb}{\mbox{$\kModel{M},w\models^{g}\varphi'\to\varphi''$}}
\settowidth{\templengthc}{\mbox{$\kModel{M},w\models^g P(x_1,\ldots,x_n)$}}
%%%%%%%%%%%%%%%%%%%%%%%%%%%%%%%%%%%%%%%%%%%%%%%%%%%%%%%%%%%%%%%%%%%%%%
$$
\begin{array}{lcl}
\parbox{\templengthc}{{}\hfill\parbox{\templengthb}{$\kModel{M},w \not\models^g \bot;$}}
  \\
\parbox{\templengthc}{{}\hfill\parbox{\templengthb}{$\kModel{M},w\models^g\varphi' \wedge \varphi''$}}
  & \leftrightharpoons
  & \parbox[t]{\templength}{$\kModel{M},w\models^g\varphi'$ and $\kModel{M},w\models^g\varphi''$;}
  \\
\parbox{\templengthc}{{}\hfill\parbox{\templengthb}{$\kModel{M},w\models^g\varphi' \vee \varphi''$}}
  & \leftrightharpoons
  & \parbox[t]{\templength}{$\kModel{M},w\models^g\varphi'$\hfill or\hfill $\kModel{M},w\models^g\varphi''$;}
  \\
\parbox{\templengthc}{{}\hfill\parbox{\templengthb}{$\kModel{M},w\models^g\varphi' \to \varphi''$}}
  & \leftrightharpoons
  & \parbox[t]{\templengthd}{$\kModel{M},\parbox{\templengtha}{$v$}\models^g\varphi'$ implies $\kModel{M},v\models^g\varphi''$, for every $v\in R(w)$;}
  \\
\parbox{\templengthc}{{}\hfill\parbox{\templengthb}{$\kModel{M},w\models^g\forall x\,\varphi'$}}
  & \leftrightharpoons
  & \parbox{\templengthd}{$\kModel{M},\parbox{\templengtha}{$v$}\models^{h}\varphi'$, for every $v\in R(w)$ and every assignment}
  \\
  &
  & \mbox{\phantom{$\kModel{M},w\models^{g'}\varphi'$, }$h$ such that $h \stackrel{x}{=} g$ and $h(x)\in D_v$;}
  \\
\parbox{\templengthc}{{}\hfill\parbox{\templengthb}{$\kModel{M},w\models^g\exists x\,\varphi'$}}
  & \leftrightharpoons
  & \parbox{\templengthd}{$\kModel{M},w\models^{h}\varphi'$, for some assignment $h$ such that $h \stackrel{x}{=} g$}
  \\
  &
  & \mbox{\phantom{$\kModel{M},w\models^{g'}\varphi'$, }and $h(x)\in D_w$.}
\end{array}
$$

Let $\kModel{M}$, $\kframe{F}_D$, $\kframe{F}$, and $\Scls{C}$ be an intuitionistic Kripke model, an intuitionistic augmented frame, an intuitionistic Kripke frame, and a class of intuitionistic augmented frames, respectively, $w$ a world of $\kModel{M}$, and $\varphi$ a formula with free variables $x_1,\ldots,x_n$; then define
%%%%%%%%%%%%%%%%%%%%%%%%%%%%%%%%%%%%%%%%%%%%%%%%%%%%%%%%%%%%%%%%%%%%%%%%%%%%%%%%%%
\settowidth{\templength}{\mbox{$\kModel{M},w\models^g P(x_1,\ldots,x_n)$}}
\settowidth{\templengtha}{\mbox{$\kModel{M},w\models^{h}\varphi'$, for every assignment $h$ such that}}
\settowidth{\templengthb}{\mbox{$w$}}
\settowidth{\templengthc}{\mbox{$\kframe{F}_D$}}
%%%%%%%%%%%%%%%%%%%%%%%%%%%%%%%%%%%%%%%%%%%%%%%%%%%%%%%%%%%%%%%%%%%%%%%%%%%%%%%%%%
$$
\begin{array}{rcl}
\Rem{
\parbox{\templength}{{}\hfill$\kModel{M},w\models \varphi$}
  & \leftrightharpoons
  & \parbox[t]{\templengthd}{$\kModel{M},w\models^g \varphi$, for every assignment $g$ such that}
  \\
  &
  & \mbox{\phantom{$\kModel{M},w\models^g \varphi$, }$g(x_1),\ldots,g(x_n)\in D_w$;}
  \\
} % for \Rem{...}  
\parbox{\templength}{{}\hfill$\kModel{M}\models \varphi$}
  & \leftrightharpoons
  & \parbox[t]{\templengthd}{$\kModel{M},w\models^g \varphi$, for every world $w$ of $\kModel{M}$ and every~$g$}
  \\
  &
  & \mbox{\phantom{$\kModel{M},w\models^g \varphi$, }such that $g(x_1),\ldots,g(x_n)\in D_w$;}
  \\
\parbox{\templength}{{}\hfill$\kframe{F}_D\models \varphi$}
  & \leftrightharpoons
  & \parbox[t]{\templengthd}{$\parbox{\templengthc}{$\kModel{M}$}\models \varphi$, for every intuitionistic model $\kModel{M}$ based on $\kframe{F}_D$;}
  \\
\parbox{\templength}{{}\hfill$\kframe{F}\models \varphi$}
  & \leftrightharpoons
  & \parbox[t]{\templengthd}{$\parbox{\templengthc}{$\kModel{M}$}\models \varphi$, for every intuitionistic model $\kModel{M}$ based on $\kframe{F}$;}
  \\
\parbox{\templength}{{}\hfill$\Scls{C}\models \varphi$}
  & \leftrightharpoons
  & \parbox[t]{\templengtha}{$\parbox{\templengthc}{$\kframe{F}_D$}\models \varphi$, for every $\kframe{F}_D\in\Scls{C}$.}
  \\
\end{array}
$$
If $\mathfrak{S}\models\varphi$, for a structure $\mathfrak{S}$ (a~model, a~frame, etc.), we say that the formula $\varphi$ is \defnotion{true} or \defnotion{valid} in (on, at)~$\mathfrak{S}$; otherwise, $\varphi$ is \defnotion{refuted} in (on, at)~$\mathfrak{S}$.
These notions, and the corresponding notations, can be extended to sets of formulas in a natural way: for a set of formulas $X$, define $\mathfrak{S}\models X$ as $\mathfrak{S}\models\varphi$, for every $\varphi\in X$.

%\Rem{
%Observe that, if $\mathfrak{F} = \langle W,R\rangle$ is an intuitionistic Kripke frame where $W$ is a singleton and $\varphi$ is an \mbox{$\lang$-formula}, then $\mathfrak{F} \Vdash \vp$ if, and only if, $\vp \in \mathbf{QCl}$.

The intuitionistic predicate logic $\logic{QInt}$ is the set of $\lang{L}$-formulas valid on every intuitionistic Kripke frame; it can also be defined through a Hilbert-style calculus with a finite set of axioms~\cite{GShS,vanDalen}. A \defnotion{superintuitionistic predicate logic} is a set of $\lang{L}$-formulas that includes $\logic{QInt}$ and is closed under Modus Ponens, Substitution, and Generalization. If $L$ is a superintuitionistic predicate logic and $\Gamma$ is a set of $\lang{L}$-formulas, then $L + \Gamma$ denotes the smallest superintuitionistic logic containing $L \cup \Gamma$. For a formula~$\varphi$, we write $L+\varphi$ rather than $L+\{\varphi\}$. If $L$ is a propositional superintuitionistic logic, then define $\logic{Q}L$ by $\logic{Q}L = \logic{QInt} + L$. 
%For a superintuitionistic logic~$L$, let $L^+$ to denote the \defnotion{positive fragment} of~$L$, i.e., the subset of~$L$ consisting of positive formulas.

Let $\Scls{C}$ be a class of intuitionistic augmented frames. Define the \defnotion{superintuitionistic predicate logic\/ $\QSIL \Scls{C}$ of the class\/~$\Scls{C}$} by
$$
\begin{array}{lcl}
\QSIL \Scls{C} & = & \{\varphi\in\lang{L} : \Scls{C}\models\varphi\}.
\end{array} 
$$
For a class $\scls{C}$ of intuitionistic Kripke frames, define 
\begin{itemize}
\item
$\aug{e}{}{\scls{C}}$ be the classes of e-augmented frames based on Kripke frames of~$\scls{C}$;
\item
$\aug{c}{}{\scls{C}}$ be the classes of c-augmented frames based on Kripke frames of~$\scls{C}$,
\end{itemize} 
and let
$$
\begin{array}{lcl}
\QSILext{e}{} \scls{C} & = & \QSIL \aug{e}{} \scls{C};
\\
\QSILext{c}{} \scls{C} & = & \QSIL \aug{c}{} \scls{C}.
\end{array} 
$$
Observe that $\QSILext{e}{} \scls{C} \subseteq \QSILext{c}{} \scls{C}$.
For an intuitionistic Kripke frame $\kframe{F}$, we write $\QSILext{e}{}\kframe{F}$ and $\QSILext{c}{}\kframe{F}$ rather than $\QSILext{e}{}\{\kframe{F}\}$ and $\QSILext{c}{}\{\kframe{F}\}$, respectively; similarly for augmented frames.

Let $P$ be a unary predicate letter and $p$ a proposition letter (i.e., nullary predicate letter); the formula $\bm{cd} = \forall x\, (P(x) \dis p) \imp \forall x\, P(x) \dis p$ is valid on an intuitionistic augmented frame $\kframe{F}_D$ if, and only if, $\kframe{F}_D$ satisfies \eqref{eq:LCD}. 
If $L$ is a superintuitionistic predicate logic, then $L\logic{.cd}$ denotes the logic $L + \bm{cd}$. 
%Recall that $\logic{QKC} = \logic{QInt} + \neg p\vee \neg\neg p$.

Here, we are interested in logic $\logic{QLC}$~--- more exactly, in its two-variable fragment~--- and some its extensions, such as the logics of the frames $\otuple{\numZ,\leqslant}$ and $\otuple{\numR,\leqslant}$. Logic $\logic{QLC}$ is a predicate counterpart of the superintuitionistic propositional logic $\logic{LC}=\logic{Int}+(p\to q)\vee(q\to p)$. It is known that $\logic{QLC}$ is the logic of the class of linear intuitionistic Kripke frames and even is the logic of the frame $\otuple{\numQ,\leqslant}$~\cite[Section~6.7]{GShS}; also, it is the `superintuitionistic' fragment of the modal predicate logic $\logic{QS4.3}$, under the G\"odel translation.

\section{Tiling problem}
\setcounter{equation}{0}

We shall prove that $\logic{QLC}$ and some its extensions are undecidable ($\Sigma^0_1$\nobreakdash-hard)
% or even highly undecidable ($\Pi^1_1$\nobreakdash-hard) 
in the language with two individual variables. To that end, we reduce a $\Pi^0_1$\nobreakdash-complete 
%or, respectively, a $\Sigma^1_1$\nobreakdash-complete 
recurrent tiling problem to the satisfiability problem for the logics.
%
%Here, we shortly recall some notions and formulate the tiling problems~\cite{Berger66,Harel86} we shall deal with.
Here, we briefly recall some notions and formulate the tiling problem that we shall deal with.

Let us think of a \defnotion{tile} as a colored $1 \times 1$ square with fixed orientation. Each edge of a tile is colored by a \defnotion{color} from a countable palette (so, we can take the colors to be natural numbers or words in a finite alphabet).
A \defnotion{tile type} consists of a specification of a color for each edge; we write $\leftsq t$, $\rightsq t$, $\upsq t$, and $\downsq t$ for the colors of the left, the right, the top, and the bottom edges of the tiles of tile type~$t$.

Let $T$ be a non-empty set of tile types. A \defnotion{$T$-tiling} is a function $f\colon\numN\times\numN\to T$. We may think of a $T$-tiling as an arrangement of tiles, whose types are in $T$, on an $\numN \times \numN$ grid. 

Let $f\colon\numN\times\numN\to T$ be a $T$-tiling; we define some conditions for it, allowing us to formulate the tiling problem.
%s.
% we shall deal with.  
%
The first condition is that the edge colors of the adjacent tiles match horizontally: for all $i, j \in \numN$,
\begin{equation}
\label{eq:T1}
\begin{array}{lcl}
\rightsq f(i,j) & = &\leftsq f(i+1,j).
\end{array}
\end{equation}
The second one is that they match vertically: for all $i, j \in \numN$,
\begin{equation}
\label{eq:T2}
\begin{array}{lcl}
\upsq f(i,j) & = & \downsq f(i,j+1).
\end{array}
\end{equation}
\Rem{ %%%%%%%%%%%%%%%%%%%%%%%%%%%%%%%%
Finally, the third one, is that there are infinitely many tiles of a specified tile type $t_0\in T$ in the left-most row of the $T$-tiling:
\begin{equation}
\label{eq:T3}
\begin{array}{lcl}
\mbox{the set $\{ j \in \numN : f(0, j) = t_0 \}$ is infinite.}
\end{array}
\end{equation}
}%%%%%%%%%%%%%%%%%%%%%%%%%%%%%%%%%%%%
The 
%first 
tiling problem we consider is the following: given a non-empty finite set $T$ of tile types, we are to determine whether there exists a $T$-tiling satisfying~\eqref{eq:T1} and~\eqref{eq:T2}. This tiling problem is known to be $\Pi^0_1$-complete~\cite{Berger66}. 
%The second one is the following: given a non-empty finite set $T$ of tile types and a tile type $t_0\in T$, we are to determine whether there exists a $T$-tiling satisfying conditions \mbox{\eqref{eq:T1}--\eqref{eq:T3}}. It is known that this tiling problem is $\Sigma^1_1$-complete~\cite{Harel86}.

\section{Technique preliminaries}
\setcounter{equation}{0}

To describe the set $\numN\times\numN$, let us take a well-known Cantor's enumeration $\rzpair \colon\numN\to\numN\times\numN$ of the set $\numN\times\numN$; the enumeration has been used by M.\,Marx and M.\,Reynolds~\cite{MarxReynolds:1999} for the same purposes in the context of compass logic~\cite{Venema:1990} (for the method, see also~\cite{RZ01} and \mbox{\cite[Chapter~7]{GKWZ}}), and we save some notations from~\cite{MarxReynolds:1999} below. The enumeration is defined by the following clauses, for all $i\in\numNp$ and $j,k\in\numN$:
\settowidth{\templengtha}{\mbox{$k$}}
\settowidth{\templengthb}{\mbox{$0$}}
\settowidth{\templengthc}{\mbox{$j$}}
\begin{flalign}
%\label{eq:pair}
&\rzpair (\parbox{\templengtha}{0}) = \otuple{0,0}; \label{eq:pair:0}\\
&\rzpair (k) = \otuple{\parbox{\templengthb}{\centering$i$},j} \quad \imply \quad \rzpair(k+1) = \otuple{\parbox{\templengthc}{\centering$i$}-1,j+1}; \label{eq:pair:2}\\
&\rzpair (k) = \otuple{0,j} \quad \imply \quad \rzpair(k+1) = \otuple{j+1,0}, \label{eq:pair:1}
\end{flalign} 
see Figure~\ref{fig:1}. Clearly, there exists the converse function $\nmbr\colon \numN\times\numN \to \numN$, i.e., defined by
$$
\begin{array}{lcl}
\nmbr(i,j) = k 
  & \iff 
  & \rzpair (k) = \otuple{i,j}. 
\end{array}
$$
For convenience, we write also $\otuple{i_k,j_k}$ instead of $\rzpair(k)$ below; in particular, $\nmbr(i_k,j_k) = k$.
Next, let us define the functions $\Rght\colon\numN\to\numN$ and $\Abv\colon\numN\to\numN$ as follows: for every $k\in\numN$, 
\begin{flalign}
&\Rght(k) ~=~ \nmbr(i_k+1,j_k); \label{eq:pair:right}\\ 
&\Abv(k)  ~=~ \nmbr(i_k,j_k+1). \label{eq:pair:above}
\end{flalign} 
Finally, in order to get a complete correspondence with the system of concepts proposed in~\cite{MarxReynolds:1999}, we define properties $\Wall$ and $\Floor$ of natural numbers by
\begin{flalign}
\Wall(k) ~~&\bydef~~ i_k=0; \label{eq:pair:wall}\\ 
\Floor(k)  ~~&\bydef~~ j_k=0. \label{eq:pair:floor}
\end{flalign} 
Notice that, for every $k\in \numN$,
\begin{equation}
\label{eq:coonection:right-above-a}
\begin{array}{lcl}
\Abv(k) 
 & =
 & \Rght(k)+1
\end{array}
\end{equation} 
and
\begin{equation}
\label{eq:coonection:right-above}
\begin{array}{lcl}
\Rght(k+1) 
 & =
 & \left\{
   \begin{array}{ll}
   \Abv(k)   & \mbox{if $\neg\Wall(k)$;}\\
   \Abv(k)+1 & \mbox{if $\phantom{\neg}\Wall(k)$.}\\
   \end{array}
   \right.
\end{array}
\end{equation} 

\begin{figure}
\centering
\begin{tikzpicture}[scale=0.96]

\foreach \x in {0,...,6}
  {
    \draw [color = black!25] (\x,-0.75)--(\x,5.45);
    \node [below, color=black!45] at (\x,-0.75) {$\x$};
    \node [above, color=black!25] at (\x, 5.45) {$\vdots$};
  }
\foreach \y in {0,...,5}
  {
    \draw [color = black!25] (-0.75,\y)--(6.45,\y);
    \node [left , color=black!45] at (-0.75,\y) {$\y$};
    \node [right, color=black!25] at ( 6.45,\y) {$\cdots$};
  }

\foreach \x in {0,...,6}
  \foreach \y in {0,...,5}
    {
      \coordinate (g\x\y) at (\x,\y);
		  \draw [fill=black!12.5] (g\x\y) circle [radius=2pt];
    }

\begin{scope}[>=latex, ->, shorten >= 2.1pt, shorten <= 2.1pt, color=black]
\draw  [] (g00)--(g10);
\draw  [] (g10)--(g01);
\draw  [] (g01)--(g20);
\draw  [] (g20)--(g11);
\draw  [] (g11)--(g02);
\draw  [] (g02)--(g30);
\draw  [] (g30)--(g21);
\draw  [] (g21)--(g12);
\draw  [] (g12)--(g03);
\draw  [color=black!75] (g03)--(g40);
\draw  [color=black!64] (g40)--(g31);
\draw  [color=black!60] (g31)--(g22);
\draw  [color=black!56] (g22)--(g13);
\draw  [color=black!52] (g13)--(g04);
\draw  [color=black!45] (g04)--(g50);
\draw  [color=black!40] (g50)--(g41);
\draw  [color=black!35] (g41)--(g32);
\draw  [color=black!30] (g32)--(g23);
\draw  [color=black!25] (g23)--(g14);
\draw  [color=black!20] (g14)--(g05);
\draw  [color=black!16] (g05)--(g60);
\draw  [color=black!12] (g60)--(g51);
\draw  [color=black!8] (g51)--(g42);
\draw  [color=black!4] (g42)--(g33);
\draw  [color=black!3] (g33)--(g24);
\end{scope}

\node [below left] at (g00) {$0$};
\node [below left] at (g10) {$1$};
\node [below left] at (g01) {$2$};
\node [below left] at (g20) {$3$};
\node [below left] at (g11) {$4$};
\node [below left] at (g02) {$5$};
\node [below left] at (g30) {$6$};
\node [below left] at (g21) {$7$};
\node [below left] at (g12) {$8$};
\node [below left] at (g03) {$9$};
\node [below left] at (g40) {$10$};
\node [below left] at (g04) {$14$};
\node [below left] at (g50) {$15$};
\node [below left] at (g05) {$20$};
\node [below left] at (g60) {$21$};

\end{tikzpicture}
\caption{Enumeration of $\numN\times\numN$}
\label{fig:1}
\end{figure}

Using functions defined by \eqref{eq:pair:0}--\eqref{eq:pair:above}, we can readily rewrite %\eqref{eq:T1}--\eqref{eq:T3} 
\eqref{eq:T1} and \eqref{eq:T2} 
with use of natural numbers as the arguments. So, \eqref{eq:T1} looks as
$$
%\begin{equation}
%\label{eq:T1:ra}
\begin{array}{lcl}
\rightsq f(\rzpair(k)) & = &\leftsq f(\rzpair(\Rght(k))\phantom{.}
%,
\end{array}
%\end{equation}
$$
and \eqref{eq:T2} looks as
$$
%\begin{equation}
%\label{eq:T2:ra}
\begin{array}{lcl}
\upsq f(\rzpair(k)) & = & \downsq f(\rzpair(\Abv(k)).
%,
\end{array}
%\end{equation}
$$
\Rem{ %%%%%%%%%%%%%%%%%%%%%%%%%%%%%%%%%%%
and condition~\eqref{eq:T3}, with use of \eqref{eq:pair:wall}, looks as
\begin{equation}
\label{eq:T3:ra}
\begin{array}{lcl}
\mbox{the set $\{ k \in \numN : \mbox{$f(\rzpair(k)) = t_0$ and $\Wall(k)$} \}$ is infinite.}
\end{array}
\end{equation}
}%%%%%%%%%%%%%%%%%%%%%%%%%%%%%%%%%%%%%%%

Although the construction presented below is based primarily on the use of different observations, it allows us to integrate the described system of concepts into~it. We shall use unary predicate letters $\rght$, $\abv$, $\wall$, $\floor$, and $\nxt$~--- the latter corresponds to the function $\Nxt\colon \numN\to\numN$ defined by $\Nxt(k)=k+1$, for every $k\in\numN$,~--- but not quite so that, for example, 
the truth of $\rght(x)$ alone makes it possible to say something definite about the value $\Rght(x)$, and the same for $\abv(x)$ and $\Abv(x)$. However, in some significant cases, a certain correspondence will be achieved.

Let us explain a technical idea that we shall use for this purpose.

To simulate a $T$-tiling, we need a tool allowing us to distinguish elements in the domains of worlds of Kripke models over linear Kripke frames. To distinguish such an element by means of the modal predicate language, we can use a unary predicate letter, say,~$U$: it is possible to make $U(a)$ being true at a world but being false at all different worlds, both accessible from it and seeing it. There is a certain difficulty to do the same using the intuitionistic predicate language: if $\kModel{M},w\models U(a)$, for some $a$ from the domain of a world $w$ of a model $\kModel{M}$, then also $\kModel{M},w'\models U(a)$, for every $w'$ accessible from~$w$. To overcome this difficulty, we shall use not one, but two (or more) unary predicate letters for such purposes. Exactly, let $U$ and $U'$ be unary predicate letters; $w'$, $w$, and $w''$ be different worlds of a model $\kModel{M}$ such that $w'$ sees $w$ and $w$ sees $w''$; $a$ be an element of the domains of the worlds. Let us assume $a$ be \defnotion{distinguished} at $w$ by means of $U$ and $U'$ if  
$$
\begin{array}{lcl}
\kModel{M},w''\hfill\models U'(a) 
  & \mbox{and} 
  & \kModel{M},w''\hfill\models U(a);
  \\
\kModel{M},w\hfill\not\models U'(a) 
  & \mbox{and} 
  & \kModel{M},w\hfill\models U(a);
  \\
\kModel{M},w'\hfill\not\models U'(a) 
  & \mbox{and} 
  & \kModel{M},w'\hfill\not\models U(a).
  \\
\end{array}
$$
If the domain of a Kripke model over a Kripke frame validating $\logic{QLC}$ is $\numN$, then, distinguishing an element $k$ with a pair of certain unary predicate letters, we can achieve that, for exapmle, $\rght(k)$ corresponds to $\Rght(k)$ or that $\abv(k)$ corresponds to $\Abv(k)$.

Let us turn to the exact constructions.

\section{Main construction}
\setcounter{equation}{0}

%\newpage

Let us define formulas describing an $\numN\times\numN$ grid. We divide them on two parts: the technical part involving the `double labeling' and the part where the `labels' are used to simulate the functions and properties defined by \eqref{eq:pair:right}--\eqref{eq:pair:floor}. We will use unary predicate letters $Q$, $Q'$, $S$, $S'$, $S''$, $G$, $\nxt$ for the `double labeling' and $\abv$, $\rght$, $\wall$, $\floor$ to simulate the corresponding functions and properties. The letter $\nxt$, used for the `labeling', will also describe the function $\Nxt\colon\numN\to\numN$ mentioned above; but this remark is insignificant for~us. 

%Let $\lang{UL}_T$ be the set consisting of all unary predicate letters we are going to use to simulate $T$-tiling, i.e., $\lang{UL}_T$ includes $Q$, $Q'$, $S$, $S'$, $S''$, $G$, $\nxt$, $\abv$, $\rght$, $\wall$, $\floor$, and $P_0,\ldots,P_n$. 

\pagebreak[3]

We start with the `double labeling'. Let $\lhd$ be a binary predicate letter. Define $\mathit{DL}$ to be a conjunction of the following formulas:
\settowidth{\templengtha}{\mbox{$\mathit{Serial}_{\lhd}$}}
%\settowidth{\templengthb}{\mbox{$\forall x\forall y\,\Big(\big(\forall x\,(x\lhd y\to \wall(x))\wedge\neg\wall(y)\wedge \rght(y)\wedge Q(x)\to Q'(x) \vee S''(y)\big)\Big);$}}
\settowidth{\templengthb}{\mbox{$\forall x\forall y\,\Big(\to
    \big(\exists y\,(y\lhd x\wedge \wall(y))\wedge \nxt(x)\wedge \exists x\,(x\lhd y \wedge \abv(x)) \wedge G(y) \to Q(x)\vee S(y)\big)\Big);$}}
$$
\begin{array}{rcl}
%\mathit{AllTiles}(x)
%  & = 
%  & \displaystyle
%    \bigwedge\limits_{\mathclap{i=0}}^{n}P_i(x); 
%  \\  
%\delta^-(x) 
%  & = 
%  & Q(x)\to Q'(x); 
%  \\
%x\preccurlyeq y 
%  & = 
%  & Q(y)\to Q(x); 
%  \smallskip\\
%A_1^k(x) 
%  & = 
%  & \exists y\,(x\lhd y \wedge P_k(y)) \to \forall y\,(x\lhd y\to P_k(y)); 
%  \\
  
\Rem{ %%%%%%%%%%%%%%%%%%%%%%%%%%%%%%%
\mathit{primes}_1 
  & = 
  & \forall x\,\Big(\big(Q'(x)\to Q(x)\big)\wedge \big(Q(x)\to \nxt(x)\big)\Big); 
  \smallskip\\
\mathit{primes}_2 
  & = 
  & \forall x\,\Big(\big(S''(x)\to S'(x)\big)\wedge \big(S'(x)\to S(x)\big)\wedge \big(S(x)\to G(x)\big) \Big); 
  \smallskip\\
} %%%%%%%%%%%%%%%%%%%%%%%%%%%%%%%%%%%
\parbox{\templengtha}{{}\hfill$\mathit{Serial}_{\lhd}$}  
  & = 
  & \parbox{\templengthb}{$\forall x\exists y\,(x\lhd y);$} 
  \smallskip\\
\Rem{ %%%%%%%%%%%%%%%%%%%%%%%%%%%%%%%
\mathit{Inclusion}_{\lhd}^{\preccurlyeq}  
  & = 
  & \forall x\forall y\,(x\lhd y\to x\preccurlyeq y); 
  \smallskip\\
\mathit{Antisym}_{\lhd}  
  & = 
  & \forall x\forall y\,\big(x\lhd y\to \neg(y\lhd x)\big); 
  \smallskip\\
} %%%%%%%%%%%%%%%%%%%%%%%%%%%%%%%%%%%
\mathit{Diag}_{N}  
  & = 
  & \forall x\forall y\,\Big(x\lhd y\to \big(Q(x) \lra \nxt(y)\big)\Big); 
  \smallskip\\
\mathit{Diag}_{Q}  
  & = 
  & \forall x\forall y\,\Big(x\lhd y\to \big(Q'(x)\lra Q(y)\big)\Big); 
  \smallskip\\
\mathit{Diag}_{S}  
  & = 
  & \forall x\forall y\,\Big(x\lhd y\to \big(S'(x) \lra S(y)\big) \wedge \big(S''(x)\lra S'(y)\big)\Big); 
  \smallskip\\
\mathit{Diag}_{G}  
  & = 
  & \forall x\forall y\,\Big(x\lhd y\to \big(S(x) \lra G(y)\big)\Big); 
  \smallskip\\
\mathit{Agree}_{S} 
  & = 
  & \displaystyle
    \forall x\forall y\, 
      \Big( 
        \big(Q(x)\wedge S(y)\to Q'(x)\vee S'(y)\big) 
        \vee 
        \big(Q(x)\wedge S'(y)\to Q'(x)\vee S''(y)\big) 
      \Big); 
  \smallskip\\
\mathit{Agree}_{G} 
  & = 
  & \displaystyle
    \forall x\forall y\, 
      \Big( 
        \big(Q(x)\wedge G(y)\to Q'(x)\vee S(y)\big) 
        \vee 
        \big(Q(x)\wedge S'(y)\to Q'(x)\vee S''(y)\big) 
      \Big); 
  \smallskip\\
\mathit{Agree}_\lhd  
  & = 
  & \displaystyle
    \forall x\forall y\, \big(y\lhd x\wedge S(x)\to S(y)\big). 
  \smallskip\\
\end{array}
$$

Next, let as describe connections of the `labeling' with the functions and properties defined by \eqref{eq:pair:right}--\eqref{eq:pair:floor}. Define $\mathit{FRAW}$ to be a conjunction of the following formulas:
$$
\begin{array}{rcl}
%\mathit{Agree}'_\lhd  
%  & = 
%  & \displaystyle
%    \forall x\forall y\,\Big(Q(x)\wedge S(y)\to 
%
%                             \big(Q'(x)\to S'(y)\big) \wedge
%                             \big(Q'(x)\wedge\wall(y)\to S''(y)\big)\Big); 
%  \smallskip\\
%\mathit{Agree}_1  
%  & = 
%  & \displaystyle
%    \forall x\,\bigwedge\limits_{\mathclap{U\in\lang{UL}_T}} \Big(\exists y\,\big(x\lhd y \wedge U(y)\big) \to \forall y\,\big(x\lhd y\to U(y)\big)\Big); 
%  \smallskip\\
%\mathit{Agree}_2  
%  & = 
%  & \displaystyle
%    \forall x\,\bigwedge\limits_{\mathclap{U\in\lang{UL}_T}} \Big(\exists y\,\big(y\lhd x \wedge U(y)\big) \to \forall y\,\big(y\lhd x\to U(y)\big)\Big); 
%  \smallskip\\
%\mathit{EM}_{\lhd}  
%  & = 
%  & \forall x\forall y\,\big(x\lhd y\vee \neg (x\lhd y)\big);
%  \smallskip\\
\parbox{\templengtha}{{}\hfill$\mathit{EM}_W$}  
  & = 
  & \parbox{\templengthb}{$\forall x\,%\Big(
    \big(\wall(x)\vee \neg \wall(x)\big)$;}
    %\wedge
    %\big(\floor(x)\vee \neg \floor(x)\big)
    %\Big);
  \smallskip\\
\mathit{Conn}_1
  & =
  & \forall x\,\Big(\big(\floor(x) \to \neg \abv(x)\big)\wedge \big(\wall(x) \to \neg \rght(x)\big)\Big) ;
  \smallskip\\
\mathit{Conn}_2
  & =
  & \forall x\forall y\,\Big(x\lhd y \to \big(\rght(x)\to \abv(y)\big)\wedge \big(\wall(x)\to \floor(y)\big)\Big) ;
  \smallskip\\
\mathit{Conn}_3
  & =
  & \forall x\,\Big(\big(\abv(x) \to S(x)\big)\wedge \big(\rght(x) \to S'(x)\big)\Big) ;
  \smallskip\\
%\mathit{Start}_\lhd_1
%  & =
%  & \exists x\,(\wall(x)\wedge\floor(x));
%  \smallskip\\
\mathit{Start}_\lhd %_2
  & =
  & \forall x\forall y\,\big(x\lhd y\wedge \wall(x)\wedge\floor(x) \to \rght(y)\big);
  \smallskip\\

\mathit{Move}_1
  & =
  & \forall x\forall y\,\Big(
    \big(\forall x\,(x\lhd y\to \wall(x))\wedge\neg\wall(y)\wedge \rght(y)\wedge Q(x)
         \to Q'(x) \vee S''(y)\big)
  \phantom{\Big);}
  \\
  &
  & %\phantom{\forall x\forall y\,\Big(}  
    \to
    \big(\exists y\,(y\lhd x\wedge \wall(y))\wedge \nxt(x)\wedge \exists x\,(x\lhd y \wedge \abv(x)) \wedge G(y) 
%  \\
%  &
%  & \hfill
         \to Q(x)\vee S(y)\big)
    \Big);
  \smallskip\\
\mathit{Move}_2
  & =
  & \forall x\forall y\,\Big(
    \big(\neg\wall(y)\wedge \rght(y)\wedge Q(x)
         \to Q'(x) \vee S''(y)\big)
  \\
  &
  & %\phantom{\forall x\forall y\,\Big(}  
    \to
    \big(\exists y\,(y\lhd x \wedge \neg\wall(y))\wedge \nxt(x)\wedge \abv(y)
         \to Q(x)\vee S'(y)\big)
    \Big).
  \smallskip\\

\Rem{ %%%%%%%%%%%%%%%%%%%%%%%%%%%%%%%%%%%%%%%%%%  
\mathit{Move}_1
  & =
  & \forall x\forall y\,\Big(
    \big(\wall(y)\wedge Q(y)\to Q'(y)\big)
  \\
  &
  & \phantom{\forall x\forall y\,\Big(}  
    \to
    \big(x\preccurlyeq y\wedge \wall(x)\wedge Q(x)\wedge \abv(y)\to Q'(x)\vee S'(y)\big)
    \Big);
  \smallskip\\
\mathit{Move}_2
  & =
  & \forall x\forall y\,\Big(
    \big(\rght(y)\wedge Q(y)\to Q'(y)\vee S'(y)\big)
  \\
  &
  & \phantom{\forall x\forall y\,\Big(}  
    \to
    \big(x\preccurlyeq y\wedge \wall(x)\wedge \floor(y)\wedge \nxt(x)\wedge G(y)\to Q(x)\vee S(y)\big)
    \Big);
  \smallskip\\
\mathit{Move}_3
  & =
  & \forall x\forall y\,\Big(
    \big(\floor(y)\wedge Q(y)\to Q'(y)\big)
  \\
  &
  & \phantom{\forall x\forall y\,\Big(}  
    \to
    \big(x\preccurlyeq y\wedge \floor(x)\wedge Q(x)\wedge \rght(y)\to Q'(x)\vee S''(y)\big)
    \Big);
  \smallskip\\
\mathit{Move}_4
  & =
  & \forall x\forall y\,\Big(
    \big(S(y)\wedge Q(x)\to Q'(x)\vee S'(y)\big)
  \\
  &
  & \phantom{\forall x\forall y\,\Big(}  
    \to
    \big(x\preccurlyeq y\wedge \nxt(x)\wedge G(y)\to Q(x)\vee S(y)\vee \wall(y)\big)
    \Big);
  \smallskip\\
\mathit{Move}_5
  & =
  & \forall x\forall y\,\Big(
    \big(S'(y)\wedge Q(x)\to Q'(x)\vee S''(y)\big)
  \\
  &
  & \phantom{\forall x\forall y\,\Big(}  
    \to
    \big(x\preccurlyeq y\wedge\wall(y)\wedge \nxt(x)\wedge G(y)\to Q(x)\vee S(y)\big)
    \Big);
  \smallskip\\
} %%%%%%%%%%%%%%%%%%%%%%%%%%%%%%%%%%%%%%%%%%%%%%  
\end{array}
$$

Then, let us define $\mathit{Grid}$\/ by
$$
\begin{array}{rcl}
\mathit{Grid} & = & \mathit{DL}\wedge\mathit{FRAW}.
\end{array}
$$

The next step is to describe a $T$-tiling, with $T=\{t_0,\ldots,t_n\}$, satisfying \eqref{eq:T1} and~\eqref{eq:T2}. To that end, let us use new unary predicate letters $P_0,\ldots,P_n$. Define $\mathit{Tiling}_T$ to be a conjunction of the following formulas:
$$
\begin{array}{rcl}
\parbox{\templengtha}{{}\hfill$T_0$} 
  & = 
  & \parbox{\templengthb}{$\displaystyle\forall x\,\bigvee\limits_{\mathclap{i=0}}^{n}\Big(P_i(x) \wedge 
               \bigwedge\limits_{\mathclap{j\ne i}}^{n}\neg P_j(x) \Big)$;}
  \smallskip\\
%??~~ T_0 
%  & = 
%  & \displaystyle
%    \forall x\,\bigvee\limits_{\mathclap{i=0}}^{n}P_i(x) \wedge 
%    \forall x\,\Big(\bigvee\limits_{\mathclap{i\ne j}}\big(P_i(x)\wedge P_j(x)\big) \to
%                    \bigwedge\limits_{\mathclap{i=0}}^{n}P_i(x) \Big);
%  \smallskip\\
%\mathit{refute}
%  & = 
%  & \displaystyle
%    \exists x\,\Big(Q(x)\wedge \bigvee\limits_{\mathclap{i=0}}^{n}P_i(x) 
%                \to 
%                Q'(x)\vee \bigwedge\limits_{\mathclap{i=0}}^{n}P_i(x)\Big);
%  \smallskip\\
T_1 
  & = 
  & \displaystyle
    \forall x\forall y\,\bigwedge\limits_{\mathclap{i=0}}^{n}\Big(
    \bigvee \{P_j(y) : \rightsq t_i\ne \leftsq t_j\}
%  \\
%  &
%  & \phantom{\forall x\forall y\,\bigwedge\limits_{\mathclap{i=0}}^{n} \Big(} 
    \to
    \big(
%      x\preccurlyeq y\wedge 
      \rght(y) \wedge Q(x)\wedge P_i(x) 
      \to Q'(x)\vee S''(y)
    \big)
    \Big);
  \smallskip\\
T_2 
  & = 
  & \displaystyle
    \forall x\forall y\,\bigwedge\limits_{\mathclap{i=0}}^{n}\Big(
    \bigvee \{P_j(y) : \upsq t_i\ne \downsq t_j\}
%  \\
%  &
%  & \phantom{\forall x\forall y\,\bigwedge\limits_{\mathclap{i=0}}^{n} \Big(} 
    \to
    \big(
%      x\preccurlyeq y\wedge 
      \abv(y) \wedge Q(x)\wedge P_i(x) 
      \to Q'(x)\vee S'(y)
    \big)
    \Big).
  \smallskip\\
\end{array}
$$

And the final touch of the construction is the formula
$$
\begin{array}{lcl}
\mathit{Refute}
  & = 
  & \forall x\forall y\,\big( 
                x\lhd y\wedge \wall(x)\wedge \floor(x)\wedge Q(x) %\wedge S'(y)
%  \\
%  &
%  & \phantom{\forall x\forall y\,\Big(}                
                \to 
                Q'(x)\vee S''(y) %\vee \mathit{refute}
                \big).
  \smallskip\\
\end{array}
$$

Now, we are ready to define a formula $\varphi_T$ whose refutability on the class of linear Kripke frames, as we will see below, means that there exists a $T$-tiling satisfying \eqref{eq:T1} and \eqref{eq:T2}.~Let
$$
\begin{array}{lcl}
\varphi_T & = & \mathit{Grid}\wedge \mathit{Tiling}_T \to \mathit{Refute}.
\end{array}
$$

\Rem{%%%%%%%%%%%%%%%%%%%%%%%%%%%%%%%%%%%%%%%%%%%%%%%%%%
\begin{figure}
\centering
\begin{tikzpicture}[scale=1.64]

\foreach \y in {0,...,4}
{
  \draw [] (-1.25,\y+0.5) circle [radius=0.175];
  \node [left] at (-1.32,\y+0.36) {${\y}$};
  \draw [>=latex,->] (-1.25,\y+0.7)--(-1.25,\y+1.3);
}
\node [] at (-1.25,5.96) {$\vdots$};
\foreach \x in {0,...,4} \node [] at (\x+0.5,5.96) {$\vdots$};
\foreach \y in {0,...,4} \node [] at (5.96,\y+0.5) {$\cdots$};
  
\foreach \y in {0,...,3}
  \foreach \x in {0,...,\y}
    \filldraw [black!2.5] (\x+0.1,1+\y+0.1)--(\x+0.1,1+\y+0.9)--(\x+0.9,1+\y+0.9)--(\x+0.9,1+\y+0.1)--cycle;

%\filldraw [red!12.5] (1.6,0.1)--(1.6,22.9)--(1.9,22.9)--(1.9,0.1)--cycle;
%\filldraw [red!12.5] (3.6,1.1)--(3.6,22.9)--(3.9,22.9)--(3.9,1.1)--cycle;
%\filldraw [red!12.5] (4.6,2.1)--(4.6,22.9)--(4.9,22.9)--(4.9,2.1)--cycle;
%\filldraw [red!12.5] (6.6,3.1)--(6.6,22.9)--(6.9,22.9)--(6.9,3.1)--cycle;
%\filldraw [red!12.5] (7.6,4.1)--(7.6,22.9)--(7.9,22.9)--(7.9,4.1)--cycle;
%\filldraw [red!12.5] (8.6,5.1)--(8.6,22.9)--(8.9,22.9)--(8.9,5.1)--cycle;

\foreach \x in {0,...,5}
  \draw [black!25] (\x,-0.75)--(\x,5);
\foreach \y in {0,...,5}
  \draw [black!25] (-0.75,\y)--(5,\y);
\foreach \x in {0,...,4}
  \node [] at (\x+0.5,-1) {${\x}$};
  
\foreach \x in {0,...,4}
  \draw [black] (\x+0.1,\x+0.1)--(\x+0.1,\x+0.9)--(\x+0.9,\x+0.9)--(\x+0.9,\x+0.1)--cycle;

\node [] at (0.5,2.5) 
  {\begin{tabular}{l}
   \begin{tabular}{llll}
    $S''$ & \!\!\!\!\!$S'$ & \!\!\!\!\!$S$ & \!\!\!\!\!$G$ \\
          & \!\!\!\!\!$Q'$  & \!\!\!\!\!$Q$  & \!\!\!\!\!n   \\
    w   & \!\!\!\!\!f  & \!\!\!\!\!r & \!\!\!\!\!a \\
   \end{tabular} \\
   ~\,$P_{f(\rzpair(0))}$
   \end{tabular}
  };

\node [] at (0.5,0.5) {wf};
\node [] at (1.5,0.5) {fr};
\node [] at (2.5,0.5) {\phantom{f}wa\phantom{f}};
\node [] at (3.5,0.5) {f};
%\node [] at (5.5,0.5) {\phantom{f}w\phantom{f}};
%\node [] at (6.5,0.5) {f};
%\node [] at (9.5,0.5) {\phantom{f}w\phantom{f}};
%\node [] at (10.5,0.5) {f};
%\node [] at (14.5,0.5) {\phantom{f}w\phantom{f}};
%\node [] at (15.5,0.5) {f};

\Rem{
\node [] at (0.5,0.5) {wf};
\node [] at (1.5,1.5) {f};
\node [] at (2.5,2.5) {w};
\node [] at (3.5,3.5) {f};
%\node [] at (5.5,5.5) {w};
%\node [] at (6.5,6.5) {f};
%\node [] at (9.5,9.5) {w};
%\node [] at (10.5,10.5) {f};
%\node [] at (14.5,14.5) {w};
%\node [] at (15.5,15.5) {f};
%\node [] at (20.5,20.5) {w};
%\node [] at (21.5,21.5) {f};
}

%\node [] at (1.5,0.5) {r};
%\node [] at (2.5,0.5) {a};
\node [] at (3.5,1.5) {r};
\node [] at (4.5,1.5) {a};
\node [] at (4.5,2.5) {r};
%\node [] at (5.5,2.5) {a};
%\node [] at (6.5,3.5) {r};
%\node [] at (7.5,3.5) {a};
%\node [] at (7.5,4.5) {r};
%\node [] at (8.5,4.5) {a};
%\node [] at (8.5,5.5) {r};
%\node [] at (9.5,5.5) {a};
%\node [] at (10.5,6.5) {r};
%\node [] at (11.5,6.5) {a};
%\node [] at (11.5,7.5) {r};
%\node [] at (12.5,7.5) {a};
%\node [] at (12.5,8.5) {r};
%\node [] at (13.5,8.5) {a};
%\node [] at (13.5,9.5) {r};
%\node [] at (14.5,9.5) {a};
%\node [] at (15.5,10.5) {r};
%\node [] at (16.5,10.5) {a};
%\node [] at (16.5,11.5) {r};
%\node [] at (17.5,11.5) {a};
%\node [] at (17.5,12.5) {r};
%\node [] at (18.5,12.5) {a};
%\node [] at (18.5,13.5) {r};
%\node [] at (19.5,13.5) {a};
%\node [] at (19.5,14.5) {r};
%\node [] at (20.5,14.5) {a};
%\node [] at (21.5,15.5) {r};
%\node [] at (22.5,15.5) {a};
%\node [] at (22.5,16.5) {r};

\Rem{
\node [] at (0.5,-0.5) {\gw};
\node [] at (0.5,-1.5) {\gf};
\node [] at (0.5,-2.5) {\rr};
\node [] at (0.5,-3.5) {\ra};
\node [] at (1.5,-0.5) {\rw};
\node [] at (1.5,-1.5) {\gf};
\node [] at (1.5,-2.5) {\gr};
\node [] at (1.5,-3.5) {\ra};
\node [] at (2.5,-0.5) {\gw};
\node [] at (2.5,-1.5) {\rf};
\node [] at (2.5,-2.5) {\rr};
\node [] at (2.5,-3.5) {\ga};
\node [] at (3.5,-0.5) {\rw};
\node [] at (3.5,-1.5) {\gf};
\node [] at (3.5,-2.5) {};
\node [] at (3.5,-3.5) {\ra};
\node [] at (4.5,-0.5) {\rw};
\node [] at (4.5,-1.5) {\rf};
\node [] at (4.5,-2.5) {};
\node [] at (4.5,-3.5) {};
\node [] at (5.5,-0.5) {\gw};
\node [] at (5.5,-1.5) {\rf};
\node [] at (5.5,-2.5) {\rr};
\node [] at (5.5,-3.5) {};
\node [] at (6.5,-0.5) {\rw};
\node [] at (6.5,-1.5) {\gf};
\node [] at (6.5,-2.5) {};
\node [] at (6.5,-3.5) {\ra};
\node [] at (7.5,-0.5) {\rw};
\node [] at (7.5,-1.5) {\rf};
\node [] at (7.5,-2.5) {};
\node [] at (7.5,-3.5) {};
\node [] at (8.5,-0.5) {\rw};
\node [] at (8.5,-1.5) {\rf};
\node [] at (8.5,-2.5) {};
\node [] at (8.5,-3.5) {};
\node [] at (9.5,-0.5) {\gw};
\node [] at (9.5,-1.5) {\rf};
\node [] at (9.5,-2.5) {\rr};
\node [] at (9.5,-3.5) {};
\node [] at (10.5,-0.5) {\rw};
\node [] at (10.5,-1.5) {\gf};
\node [] at (10.5,-2.5) {};
\node [] at (10.5,-3.5) {\ra};
\node [] at (11.5,-0.5) {\rw};
\node [] at (11.5,-1.5) {\rf};
\node [] at (11.5,-2.5) {};
\node [] at (11.5,-3.5) {};
\node [] at (12.5,-0.5) {\rw};
\node [] at (12.5,-1.5) {\rf};
\node [] at (12.5,-2.5) {};
\node [] at (12.5,-3.5) {};
\node [] at (13.5,-0.5) {\rw};
\node [] at (13.5,-1.5) {\rf};
\node [] at (13.5,-2.5) {};
\node [] at (13.5,-3.5) {};
\node [] at (14.5,-0.5) {\gw};
\node [] at (14.5,-1.5) {\rf};
\node [] at (14.5,-2.5) {\rr};
\node [] at (14.5,-3.5) {};
\node [] at (15.5,-0.5) {\rw};
\node [] at (15.5,-1.5) {\gf};
\node [] at (15.5,-2.5) {};
\node [] at (15.5,-3.5) {\ra};
\node [] at (16.5,-0.5) {\rw};
\node [] at (16.5,-1.5) {\rf};
\node [] at (16.5,-2.5) {};
\node [] at (16.5,-3.5) {};
\node [] at (17.5,-0.5) {\rw};
\node [] at (17.5,-1.5) {\rf};
\node [] at (17.5,-2.5) {};
\node [] at (17.5,-3.5) {};
\node [] at (18.5,-0.5) {\rw};
\node [] at (18.5,-1.5) {\rf};
\node [] at (18.5,-2.5) {};
\node [] at (18.5,-3.5) {};
\node [] at (19.5,-0.5) {\rw};
\node [] at (19.5,-1.5) {\rf};
\node [] at (19.5,-2.5) {};
\node [] at (19.5,-3.5) {};
\node [] at (20.5,-0.5) {\gw};
\node [] at (20.5,-1.5) {\rf};
\node [] at (20.5,-2.5) {\rr};
\node [] at (20.5,-3.5) {};
\node [] at (21.5,-0.5) {\rw};
\node [] at (21.5,-1.5) {\gf};
\node [] at (21.5,-2.5) {};
\node [] at (21.5,-3.5) {\ra};
\node [] at (22.5,-0.5) {\rw};
\node [] at (22.5,-1.5) {\rf};
\node [] at (22.5,-2.5) {};
\node [] at (22.5,-3.5) {};
}
\end{tikzpicture}
\caption{Simulating a tiling}
\label{fig:3}
\end{figure}
}%%%%%%%%%%%%%%%%%%%%%%%%%%%%%%%%%%%%%%%%%%%%%%%%%%%

Let $\kframe{G}=\otuple{\numN,\leqslant}$.

\begin{lemma}
\label{lem:1:main}
If there exists a\/ $T$-tiling satisfying\/~\eqref{eq:T1}\/ and\/ \eqref{eq:T2}, then\/ $\kframe{G}\odot\numN,0\not\models\varphi_T$.
\end{lemma}

\begin{proof}
Let $f\colon \numN\times\numN\to T$ be a $T$-tiling. Let $\kModel{M}=\otuple{\kframe{G}\odot\numN, I}$ be a model over frame $\kframe{G}\odot\numN$ satisfying, for all $w,a,b\in\numN$ and $k\in\{0,\ldots,n\}$, the following conditions:
$$
\begin{array}{lcl}
\kModel{M},w\models a\lhd b & \iff & b=a+1;
\smallskip\\
\kModel{M},w\models Q(a) & \iff & a\leqslant w;
\smallskip\\
\kModel{M},w\models Q'(a) & \iff & a \leqslant w-1;
\smallskip\\
\kModel{M},w\models \nxt(a) & \iff & a \leqslant w+1;
\smallskip\\
\kModel{M},w\models S(a) & \iff & a\leqslant\nmbr(i_w,j_w+1);
\smallskip\\
\kModel{M},w\models S'(a) & \iff & a\leqslant\nmbr(i_w,j_w+1)-1;
\smallskip\\
\kModel{M},w\models S''(a) & \iff & a\leqslant\nmbr(i_w,j_w+1)-2;
\smallskip\\
\kModel{M},w\models G(a) & \iff & a\leqslant\nmbr(i_w,j_w+1)+1;
\smallskip\\
\kModel{M},w\models \wall(a) & \iff & i_a = 0;
\smallskip\\
\kModel{M},w\models \floor(a) & \iff & j_a = 0;
\smallskip\\
\kModel{M},w\models \abv(a) & \iff & \mbox{$\kModel{M},w\models S(a)\phantom{'}$ and $\kModel{M},w\not\models \floor(a)$;}
\smallskip\\
\kModel{M},w\models \rght(a) & \iff & \mbox{$\kModel{M},w\models S'(a)$ and $\kModel{M},w\not\models \wall(a)$;}
\smallskip\\
\kModel{M},w\models P_k(a) & \iff & f(i_a,j_a) = t_k,
\smallskip\\
\end{array}
$$
see Figure~\ref{fig:2}. 

Let us give some comments to the figure.
The worlds are depicted as circles, the accessibility relation is indicated by arrows. Each cell corresponds to an individual in the domain of the world opposite which it is located. The letters `a', `f', `r', and~`w' inside a cell mean that the corresponding element, as well as all the elements above, satisfy the properties $\abv$, $\floor$, $\rght$, and $\wall$, respectively. The elements corresponding to the cells on the diagonal, as well as those above them, satisfy the property~$Q$. The ones above the diagonal satisfy the property~$Q'$. The elements immediately to the right of the diagonal ones, as well as the diagonal ones and those to the left or above of them, satisfy the property~$\nxt$. The elements corresponding to the cells with the letter~`r', as well as those to the left or above of them, satisfy the property~$S'$. The elements to the left of those marked with the letter~`r' satisfy also the property~$S''$, and the same for the elements above. The elements marked with the letter~`a' and those to the left or above of them satisfy the property~$S$. Finally, the elements immediately to the right of those marked with the letter~`a', as well as those marked with the letter~`a', and those to the left of them, satisfy the property~$G$, and the same for the elements above.
\Rem{ %%%%%%%%%%%%%%%%%%%%%%%%%%%%%%
Observe that, by the definition for $\preccurlyeq$, we readily obtain also that
$$
\begin{array}{lcl}
\kModel{M},w\models a\preccurlyeq b 
  & \iff 
  & \mbox{either $a\leqslant b$ or both $a<w$ and $b<w$.}
\smallskip\\
\end{array}
$$
}%%%%%%%%%%%%%%%%%%%%%%%%%%%%%%%%%%

A routine check shows that then $\kModel{M},0\models\mathit{Grid}\wedge \mathit{Tiling}_T$ and $\kModel{M},0\not\models \mathit{Refute}$; we leave the details to the reader (hint: for `long' implications~--- as in $\mathit{Move}_1$, $\mathit{Move}_2$, $\mathit{T}_1$ or $\mathit{T}_2$~--- just suppose that the conclusion of the implication is refuted at a world and then show that the premise is refuted at it as well). Therefore, $\kModel{M},0\not\models\varphi_T$. 
\end{proof}

\begin{figure}
\centering
\begin{tikzpicture}[scale=0.57]

\foreach \y in {0,...,18}
{
  \draw [] (-1.25,\y+0.5) circle [radius=0.175];
  \node [left] at (-1.32,\y+0.36) {${\y}$};
  \draw [>=latex,->] (-1.25,\y+0.7)--(-1.25,\y+1.3);
}
\node [] at (-1.25,19.96) {$\vdots$};
\foreach \x in {0,...,18} \node [] at (\x+0.5,19.96) {$\vdots$};
\foreach \y in {0,...,18} \node [] at (19.96,\y+0.5) {$\cdots$};
  
%\foreach \y in {0,...,17}
%  \foreach \x in {0,...,\y}
%    \filldraw [black!2.5] (\x+0.1,1+\y+0.1)--(\x+0.1,1+\y+0.9)--(\x+0.9,1+\y+0.9)--(\x+0.9,1+\y+0.1)--cycle;

%\filldraw [red!12.5] (1.6,0.1)--(1.6,22.9)--(1.9,22.9)--(1.9,0.1)--cycle;
%\filldraw [red!12.5] (3.6,1.1)--(3.6,22.9)--(3.9,22.9)--(3.9,1.1)--cycle;
%\filldraw [red!12.5] (4.6,2.1)--(4.6,22.9)--(4.9,22.9)--(4.9,2.1)--cycle;
%\filldraw [red!12.5] (6.6,3.1)--(6.6,22.9)--(6.9,22.9)--(6.9,3.1)--cycle;
%\filldraw [red!12.5] (7.6,4.1)--(7.6,22.9)--(7.9,22.9)--(7.9,4.1)--cycle;
%\filldraw [red!12.5] (8.6,5.1)--(8.6,22.9)--(8.9,22.9)--(8.9,5.1)--cycle;

\foreach \x in {0,...,19}
  \draw [black!25] (\x,-0.75)--(\x,19);
\foreach \y in {0,...,19}
  \draw [black!25] (-0.75,\y)--(19,\y);
\foreach \x in {0,...,18}
  \node [] at (\x+0.5,-1) {${\x}$};
  
\foreach \x in {0,...,18}
  \draw [black] (\x+0.1,\x+0.1)--(\x+0.1,\x+0.9)--(\x+0.9,\x+0.9)--(\x+0.9,\x+0.1)--cycle;

\node [] at (0.5,0.5) {wf};
\node [] at (1.5,0.5) {fr};
\node [] at (2.5,0.5) {\phantom{f}wa\phantom{f}};
\node [] at (3.5,0.5) {f};
\node [] at (5.5,0.5) {\phantom{f}w\phantom{f}};
\node [] at (6.5,0.5) {f};
\node [] at (9.5,0.5) {\phantom{f}w\phantom{f}};
\node [] at (10.5,0.5) {f};
\node [] at (14.5,0.5) {\phantom{f}w\phantom{f}};
\node [] at (15.5,0.5) {f};

\Rem{
\node [] at (0.5,0.5) {wf};
\node [] at (1.5,1.5) {f};
\node [] at (2.5,2.5) {w};
\node [] at (3.5,3.5) {f};
\node [] at (5.5,5.5) {w};
\node [] at (6.5,6.5) {f};
\node [] at (9.5,9.5) {w};
\node [] at (10.5,10.5) {f};
\node [] at (14.5,14.5) {w};
\node [] at (15.5,15.5) {f};
%\node [] at (20.5,20.5) {w};
%\node [] at (21.5,21.5) {f};
}

%\node [] at (1.5,0.5) {r};
%\node [] at (2.5,0.5) {a};
\node [] at (3.5,1.5) {r};
\node [] at (4.5,1.5) {a};
\node [] at (4.5,2.5) {r};
\node [] at (5.5,2.5) {a};
\node [] at (6.5,3.5) {r};
\node [] at (7.5,3.5) {a};
\node [] at (7.5,4.5) {r};
\node [] at (8.5,4.5) {a};
\node [] at (8.5,5.5) {r};
\node [] at (9.5,5.5) {a};
\node [] at (10.5,6.5) {r};
\node [] at (11.5,6.5) {a};
\node [] at (11.5,7.5) {r};
\node [] at (12.5,7.5) {a};
\node [] at (12.5,8.5) {r};
\node [] at (13.5,8.5) {a};
\node [] at (13.5,9.5) {r};
\node [] at (14.5,9.5) {a};
\node [] at (15.5,10.5) {r};
\node [] at (16.5,10.5) {a};
\node [] at (16.5,11.5) {r};
\node [] at (17.5,11.5) {a};
\node [] at (17.5,12.5) {r};
\node [] at (18.5,12.5) {a};
\node [] at (18.5,13.5) {r};
%\node [] at (19.5,13.5) {a};
%\node [] at (19.5,14.5) {r};
%\node [] at (20.5,14.5) {a};
%\node [] at (21.5,15.5) {r};
%\node [] at (22.5,15.5) {a};
%\node [] at (22.5,16.5) {r};

\Rem{
\node [] at (0.5,-0.5) {\gw};
\node [] at (0.5,-1.5) {\gf};
\node [] at (0.5,-2.5) {\rr};
\node [] at (0.5,-3.5) {\ra};
\node [] at (1.5,-0.5) {\rw};
\node [] at (1.5,-1.5) {\gf};
\node [] at (1.5,-2.5) {\gr};
\node [] at (1.5,-3.5) {\ra};
\node [] at (2.5,-0.5) {\gw};
\node [] at (2.5,-1.5) {\rf};
\node [] at (2.5,-2.5) {\rr};
\node [] at (2.5,-3.5) {\ga};
\node [] at (3.5,-0.5) {\rw};
\node [] at (3.5,-1.5) {\gf};
\node [] at (3.5,-2.5) {};
\node [] at (3.5,-3.5) {\ra};
\node [] at (4.5,-0.5) {\rw};
\node [] at (4.5,-1.5) {\rf};
\node [] at (4.5,-2.5) {};
\node [] at (4.5,-3.5) {};
\node [] at (5.5,-0.5) {\gw};
\node [] at (5.5,-1.5) {\rf};
\node [] at (5.5,-2.5) {\rr};
\node [] at (5.5,-3.5) {};
\node [] at (6.5,-0.5) {\rw};
\node [] at (6.5,-1.5) {\gf};
\node [] at (6.5,-2.5) {};
\node [] at (6.5,-3.5) {\ra};
\node [] at (7.5,-0.5) {\rw};
\node [] at (7.5,-1.5) {\rf};
\node [] at (7.5,-2.5) {};
\node [] at (7.5,-3.5) {};
\node [] at (8.5,-0.5) {\rw};
\node [] at (8.5,-1.5) {\rf};
\node [] at (8.5,-2.5) {};
\node [] at (8.5,-3.5) {};
\node [] at (9.5,-0.5) {\gw};
\node [] at (9.5,-1.5) {\rf};
\node [] at (9.5,-2.5) {\rr};
\node [] at (9.5,-3.5) {};
\node [] at (10.5,-0.5) {\rw};
\node [] at (10.5,-1.5) {\gf};
\node [] at (10.5,-2.5) {};
\node [] at (10.5,-3.5) {\ra};
\node [] at (11.5,-0.5) {\rw};
\node [] at (11.5,-1.5) {\rf};
\node [] at (11.5,-2.5) {};
\node [] at (11.5,-3.5) {};
\node [] at (12.5,-0.5) {\rw};
\node [] at (12.5,-1.5) {\rf};
\node [] at (12.5,-2.5) {};
\node [] at (12.5,-3.5) {};
\node [] at (13.5,-0.5) {\rw};
\node [] at (13.5,-1.5) {\rf};
\node [] at (13.5,-2.5) {};
\node [] at (13.5,-3.5) {};
\node [] at (14.5,-0.5) {\gw};
\node [] at (14.5,-1.5) {\rf};
\node [] at (14.5,-2.5) {\rr};
\node [] at (14.5,-3.5) {};
\node [] at (15.5,-0.5) {\rw};
\node [] at (15.5,-1.5) {\gf};
\node [] at (15.5,-2.5) {};
\node [] at (15.5,-3.5) {\ra};
\node [] at (16.5,-0.5) {\rw};
\node [] at (16.5,-1.5) {\rf};
\node [] at (16.5,-2.5) {};
\node [] at (16.5,-3.5) {};
\node [] at (17.5,-0.5) {\rw};
\node [] at (17.5,-1.5) {\rf};
\node [] at (17.5,-2.5) {};
\node [] at (17.5,-3.5) {};
\node [] at (18.5,-0.5) {\rw};
\node [] at (18.5,-1.5) {\rf};
\node [] at (18.5,-2.5) {};
\node [] at (18.5,-3.5) {};
\node [] at (19.5,-0.5) {\rw};
\node [] at (19.5,-1.5) {\rf};
\node [] at (19.5,-2.5) {};
\node [] at (19.5,-3.5) {};
\node [] at (20.5,-0.5) {\gw};
\node [] at (20.5,-1.5) {\rf};
\node [] at (20.5,-2.5) {\rr};
\node [] at (20.5,-3.5) {};
\node [] at (21.5,-0.5) {\rw};
\node [] at (21.5,-1.5) {\gf};
\node [] at (21.5,-2.5) {};
\node [] at (21.5,-3.5) {\ra};
\node [] at (22.5,-0.5) {\rw};
\node [] at (22.5,-1.5) {\rf};
\node [] at (22.5,-2.5) {};
\node [] at (22.5,-3.5) {};
}
\end{tikzpicture}
\caption{Simulating a tiling}
\label{fig:2}
\end{figure}

\begin{lemma}
\label{lem:2:main}
Let\/ $\kframe{F}=\otuple{W,R}$ be a Kripke frame validating\/ $\logic{QLC}$ such that $\kframe{F}\not\models\varphi_T$.
Then there exists a $T$-tiling satisfying\/~\eqref{eq:T1}~and\/~\eqref{eq:T2}.
\end{lemma}

\begin{proof}
Let $\kModel{M}=\otuple{W,R,D,I}$ and $w^\ast$ be, respectively, a model over $\kframe{F}$ and a world of $\kframe{F}$ such that
%$$
\begin{equation}
\label{eq:lem:2:main:0}
\begin{array}{lcl}
\kModel{M},w^\ast \models \mathit{Grid}\wedge \mathit{Tiling}_T 
  & \mbox{and}
  & \kModel{M},w^\ast\not\models \mathit{Refute}.
\end{array}
\end{equation}
%$$

Since $\kModel{M},w^\ast\not\models \mathit{Refute}$, there exist $w_0\in R(w^\ast)$ and $a_0,a_1\in D_{w_0}$ such that 
\begin{equation}
\label{eq:lem:2:main:1}
\kModel{M},w_0\models a_0\lhd a_1\wedge \wall(a_0)\wedge \floor(a_0)\wedge Q(a_0) %\wedge S'(a_1)
\end{equation}
and
\begin{equation}
\label{eq:lem:2:main:2}
\kModel{M},w_0\not\models Q'(a_0)\vee S''(a_1).
\end{equation}
By $\mathit{Serial}_{\lhd}$, there exist also $a_2,a_3,a_4,\ldots\in D_{w_0}$ such that, for every $k\in\numN$,
\begin{equation}
\label{eq:lem:2:main:3}
\kModel{M},w_0\models a_k\lhd a_{k+1}.
\end{equation}

By $T_0$, for every $k\in\numN$ there exists a unique $s\in\{0,\ldots,n\}$ such that $\kModel{M},w_0\models P_s(a_k)$. Let us define a $T$-tiling $f\colon\numN\times\numN\to T$ as follows: for all $i,j\in\numN$ and $s\in\{0,\ldots,n\}$,
\begin{equation}
\label{eq:lem:2:main:5}
\begin{array}{lcl}
%\phantom{f(\rzpair(m))=t_k}
%\phantom{_{kk}}
f(i,j)=t_s 
  & \bydef 
  & \mbox{$\kModel{M},w_0\models P_s(a_{\nmbr(i,j)})$.}
%    \phantom{_{k}}
%    \phantom{\mbox{$\kModel{M},w_0\models P_k(a_m)$.}}
\end{array}
\end{equation}
%or, equivalently, by
%\begin{equation}
%\label{eq:lem:2:main:5a}
%\begin{array}{lcl}
%\phantom{f(i,j)=t_k}
%f(\rzpair(m))=t_k 
%  & \bydef 
%  & \mbox{$\kModel{M},w_0\models P_k(a_m)$.}
%    \phantom{\mbox{$\kModel{M},w_0\models P_k(a_{\nmbr(i,j)})$,}}
%\end{array}
%\end{equation}
We are going to show that the $T$-tiling defined by~\eqref{eq:lem:2:main:5}
%~--- or, equivqlently, by \eqref{eq:lem:2:main:6}~--- 
satisfies conditions~\eqref{eq:T1}~and\/~\eqref{eq:T2}. 

To that end, let us introduce two functions, $\Rght'\colon\numN\to\numN$ and $\Abv'\colon\numN\to\numN$, defined as follows: for every $k\in\numN$,
\settowidth{\templength}{\mbox{$\abv(a_m)$}}
\begin{flalign}
&\Rght'(k) 
   ~~=~~ 
  \min\{m\in\numN : \mbox{$\kModel{M},w_0\not\models Q(a_k)\wedge \parbox{\templength}{{\hfill}$\rght\hfill(a_m)$}\to Q'(a_k)\vee S''(a_m)$}\}; 
  \label{eq:pair:right'}
  \\ 
&\Abv'(k) 
   ~~=~~ 
  \min\{m\in\numN : \mbox{$\kModel{M},w_0\not\models Q(a_k)\wedge \abv(a_m)\to Q'(a_k)\vee S'\phantom{'}(a_m)$}\}. 
  \label{eq:pair:above'}
\end{flalign} 
Also let, for every $k\in\numN$, 
\begin{equation}
\label{eq:pair:wall'}
\begin{array}{lcl}
\Wall'(k) 
  & \bydef 
  & \kModel{M},w_0\models \wall(a_k).
\end{array}
\end{equation}

\Rem{%%%%%%%%%%%%%%%%%%%%%%%%%%%%%%%%%%%%%%%
Of course, we have to show that both the functions are well-defined. We are going to prove a stronger statement: for every $k\in\numN$,
\begin{flalign}
&\Rght'(k) ~~=~~ \Rght(k); 
  \label{eq:pair:right=right'}
  \\ 
&\Abv'(k)  ~~=~~ \Abv(k). 
  \label{eq:pair:above=above'}
\end{flalign} 
For these purposes, let us introduce, for every $r\in\numN$, two auxiliary functions, $\Rght'_r\colon\{0,\ldots,r\}\to\numN$ and $\Abv'_r\colon\{0,\ldots,r\}\to\numN$, whose definitions are similar to \eqref{eq:pair:right'} and \eqref{eq:pair:above'}: for every $k\in\{0,\dots,r\}$ and $m\in\numN$,
\settowidth{\templength}{\mbox{$\abv(a_m)$}}
\begin{flalign}
&\Rght'_r(k) = m 
  \quad \bydef \quad 
  \mbox{$\kModel{M},w_0\not\models Q(a_k)\wedge \parbox{\templength}{{\hfill}$\rght\hfill(a_m)$}\to Q'(a_k)\vee S''(a_m)$;} 
  \label{eq:pair:right'r}
  \\ 
&\Abv'_r(k)  = m 
  \quad \bydef \quad 
  \mbox{$\kModel{M},w_0\not\models Q(a_k)\wedge \abv(a_m)\to Q'(a_k)\vee S'\phantom{'}(a_m)$.} 
  \label{eq:pair:above'r}
\end{flalign} 
%Clearly, 
} %%%%%%%%%%%%%%%%%%%%%%%%%%%%%%%%%%%%%%%%%%%%%%%%%%%%%%%

\begin{sublemma}
\label{sublem:lem:2:main}
For every\/ $k\in\numN$, both\/ $\Rght'(k) = \Rght(k)$ and\/ $\Abv'(k)  = \Abv(k)$.
\end{sublemma}

\begin{proof}
We prove the statement by induction on $k$. Exactly, we prove, for every $k\in\numN$, the following three statements:
\settowidth{\templength}{\mbox{$k$}}
\begin{flalign}
&\Rght'(k) ~~=~~ \Rght(k); \label{eq:1i:sublem:lem:2:main}\\
&\Abv'(k)  ~~=~~ \Abv(k); \label{eq:2i:sublem:lem:2:main}\\
&\phantom{\textnormal{\texttt{w}}}\Wall'(\parbox{\templength}{\centering$r$}) ~~=~~ \Wall(\parbox{\templength}{\centering$r$}),\phantom{\textnormal{\texttt{w}}} \qquad \mbox{where $0\leqslant r < \Abv(k)$.} \label{eq:3i:sublem:lem:2:main}
\end{flalign} 

Induction base:\, Let us prove \eqref{eq:1i:sublem:lem:2:main}--\eqref{eq:3i:sublem:lem:2:main} for $k=0$.
By~\eqref{eq:lem:2:main:1}, $\kModel{M},w_0\models \wall(a_0)\wedge \floor(a_0)$. Hence, by $\mathit{Start}_\lhd$ and $\mathit{Conn}_2$, 
\begin{equation}
\label{eq:1:sublem:lem:2:main}
\kModel{M},w_0\models \floor(a_1)\wedge\rght(a_1)\wedge\abv(a_2).
\end{equation}
Then, by~\eqref{eq:lem:2:main:1}, \eqref{eq:lem:2:main:2}, and~\eqref{eq:1:sublem:lem:2:main},
\begin{equation}
\label{eq:2:sublem:lem:2:main}
\kModel{M},w_0\not\models Q(a_0)\wedge \rght(a_1)\to Q'(a_0)\vee S''(a_1).
\end{equation}
Notice that, in~\eqref{eq:2:sublem:lem:2:main}, $a_1$ can not be replaced by $a_0$. Indeed,  $\kModel{M},w_0\models \wall(a_0)$ implies, by $\mathit{Conn}_1$, that $\kModel{M},w_0\models \neg\rght(a_0)$; but then $\kModel{M},w_0\models Q(a_0)\wedge \rght(a_0)\to Q'(a_0)\vee S''(a_0)$. Hence,
$$
\begin{array}{lclcl}
\Rght'(0) & = & 1 & = & \Rght(0),
\end{array}
$$
i.e., \eqref{eq:1i:sublem:lem:2:main} is proved.
By~\eqref{eq:lem:2:main:2}, $\kModel{M},w_0\not\models S''(a_1)$; then $\kModel{M},w_0\not\models S'(a_2)$ by $\mathit{Diag}_S$. Applying~\eqref{eq:lem:2:main:1} and~\eqref{eq:1:sublem:lem:2:main}, we obtain
\begin{equation}
\label{eq:3:sublem:lem:2:main}
\kModel{M},w_0\not\models Q(a_0)\wedge \abv(a_2)\to Q'(a_0)\vee S'(a_2).
\end{equation}
Notice that in~\eqref{eq:3:sublem:lem:2:main}, $a_2$ can not be replaced by either $a_0$ or~$a_1$. Indeed, $\kModel{M},w_0\models \floor(a_0)\wedge \floor(a_1)$ implies, again by $\mathit{Conn}_1$, that $\kModel{M},w_0\models \neg\abv(a_0)$ and $\kModel{M},w_0\models \neg\abv(a_1)$; but then both $\kModel{M},w_0\models Q(a_0)\wedge \abv(a_0)\to Q'(a_0)\vee S'(a_0)$ and $\kModel{M},w_0\models Q(a_0)\wedge \abv(a_1)\to Q'(a_0)\vee S'(a_1)$. Hence,
$$
\begin{array}{lclcl}
\Abv'(0) & = & 2 & = & \Abv(0),
\end{array}
$$ 
i.e., \eqref{eq:2i:sublem:lem:2:main} is proved.
To prove \eqref{eq:3i:sublem:lem:2:main}, observe that $\Wall'(0)$ by~\eqref{eq:lem:2:main:1}.
% and $\Wall'(2)$ by~\eqref{eq:1:sublem:lem:2:main}. 
Also, by~\eqref{eq:1:sublem:lem:2:main}, $\kModel{M},w_0\models\rght(a_1)$; then $\kModel{M},w_0\not\models\wall(a_1)$ by $\mathit{Conn}_1$, and we obtain $\neg\Wall'(1)$. Thus, 
$$
\begin{array}{cccc}
\Wall'(0)=\Wall(0) 
  & \mbox{and\/} 
  & \Wall'(1)=\Wall(1), 
%  & \mbox{and\/} 
%  & \Wall'(2)=\Wall(2),
\end{array}
$$
i.e., \eqref{eq:3i:sublem:lem:2:main} is proved.

Induction step:\, Suppose that \eqref{eq:1i:sublem:lem:2:main}--\eqref{eq:3i:sublem:lem:2:main} hold for some $k\in\numN$; let us prove the statements for~$k+1$.
Let $\Abv'(k)=\Abv(k)=m$. By the definition for $\Abv'(k)$, i.e., by~\eqref{eq:pair:above'},
$$
\kModel{M},w_0\not\models Q(a_k)\wedge \abv(a_m)\to Q'(a_k)\vee S'(a_m).
$$
Hence, there exists a world $w\in R(w_0)$ such that
\begin{equation}
\label{eq:4:sublem:lem:2:main}
\mbox{$\kModel{M},w\models Q(a_k)\wedge \abv(a_m)$ 
\quad and \quad 
$\kModel{M},w\not\models Q'(a_k)\vee S'(a_m)$.}
\end{equation}
By $\mathit{Diag}_N$, $\mathit{Diag}_Q$, $\mathit{Diag}_S$, $\mathit{Diag}_G$, and $\mathit{Conn}_3$, it follows from~\eqref{eq:4:sublem:lem:2:main} that
\begin{flalign}
&\kModel{M},w\models \nxt(a_{k+1})\wedge S'(a_{m-1}) \wedge S(a_m) \wedge G(a_{m+1}); 
\label{eq:5:sublem:lem:2:main} 
\\
&\kModel{M},w\not\models Q(a_{k+1}) \vee S''(a_{m-1}) \vee S(a_{m+1}).
\label{eq:6:sublem:lem:2:main}
\end{flalign}
Now, let us consider two cases for~$k$: $\Wall(k)$ and $\neg\Wall(k)$.

Case $\Wall(k)$: 
In this case, by the induction hypothesis, $\Wall'(k)$, which, by~\eqref{eq:pair:wall'}, gives us $\kModel{M},w_0\models \wall(a_k)$, and hence, $\kModel{M},w\models \wall(a_k)$. Then, taking into account \eqref{eq:4:sublem:lem:2:main}--\eqref{eq:6:sublem:lem:2:main}, we readily obtain that
\begin{equation}
\label{eq:7:sublem:lem:2:main} 
\begin{array}{l}
\kModel{M},w\models
\exists y\,(y\lhd a_{k+1}\wedge \wall(y))\wedge \nxt(a_{k+1})\wedge \exists x\,(x\lhd a_{m+1} \wedge \abv(x)) \wedge G(a_{m+1}); 
\smallskip\\
\kModel{M},w\not\models Q(a_{k+1})\vee S(a_{m+1}).
\end{array}
\end{equation}
By $\mathit{Move}_1$, \eqref{eq:7:sublem:lem:2:main} implies
$$
\kModel{M},w\not\models
\forall x\,(x\lhd a_{m+1}\to \wall(x))\wedge\neg\wall(a_{m+1})\wedge \rght(a_{m+1})\wedge Q(a_{k+1})
         \to Q'(a_{k+1}) \vee S''(a_{m+1}).
$$
Hence, there exists $w'\in R(w)$ such that
\begin{equation}
\label{eq:8:sublem:lem:2:main} 
\begin{array}{l}
\kModel{M},w'\models
\wall(a_m)\wedge\neg\wall(a_{m+1})\wedge \rght(a_{m+1})\wedge Q(a_{k+1});
\smallskip\\
\kModel{M},w'\not\models Q'(a_{k+1}) \vee S''(a_{m+1}).
\end{array}
\end{equation}
Since $w_0Rw'$, \eqref{eq:8:sublem:lem:2:main} implies  
\begin{equation}
\label{eq:9:sublem:lem:2:main} 
\kModel{M},w_0\not\models Q(a_{k+1})\wedge \rght\hfill(a_{m+1})\to Q'(a_{k+1})\vee S''(a_{m+1}),
\end{equation}
and we conclude, by \eqref{eq:pair:right'}, that $\Rght'(k+1)\leqslant m+1$.
To prove that $\Rght'(k+1) = m+1$, suppose, for the sake of contradiction, that $\Rght'(k+1) = m' < m+1$. Then, for some $w''\in R(w_0)$,
\begin{equation}
\label{eq:10:sublem:lem:2:main} 
\begin{array}{lcl}
\kModel{M},w''\models
\rght(a_{m'})\wedge Q(a_{k+1})
  & \mbox{and}
  & \kModel{M},w''\not\models Q'(a_{k+1}) \vee S''(a_{m'}).
\end{array}
\end{equation}
\Rem{ %%%%%%%%%%%%%%%%%%%%%%%%%%%%%%%%%%%%%%%%
Let us observe that $w''\not\in R(w')$. Indeed, by~\eqref{eq:8:sublem:lem:2:main}, $\kModel{M},w'\models\rght(a_{m+1})$; hence, $\kModel{M},w'\models S'(a_{m+1})$ by $\mathit{Conn}_3$. Then, $\kModel{M},w'\models S''(a_{m})$ by $\mathit{Diag}_S$. Applying $\mathit{Agree}_\lhd$, we obtain $\kModel{M},w'\models S''(a_{m'})$. Then, assuming $w''\in R(w')$, we have to conclude that $\kModel{M},w''\models S''(a_{m'})$, which contradicts \eqref{eq:10:sublem:lem:2:main}. Thus, $w''\not\in R(w')$, and hence, 
\begin{equation}
\label{eq:11:sublem:lem:2:main} 
w'\in R(w'').
\end{equation}
} %%%%%%%%%%%%%%%%%%%%%%%%%%%%%%%%%%%%%%%%%%%%
Clearly, $w\not\in R(w'')$: it is sufficient to observe that $\kModel{M},w\not\models Q(a_{k+1})$ by \eqref{eq:8:sublem:lem:2:main}, while $\kModel{M},w''\models Q(a_{k+1})$ by \eqref{eq:10:sublem:lem:2:main}.
Hence, 
\begin{equation}
\label{eq:12:sublem:lem:2:main} 
w''\in R(w).
\end{equation}
By $\mathit{Diag}_N$, $\mathit{Diag}_Q$, $\mathit{Diag}_S$, $\mathit{Diag}_G$, and $\mathit{Conn}_3$, it follows from~\eqref{eq:10:sublem:lem:2:main} that
\begin{flalign}
&\kModel{M},w''\models S''(a_{m'-1}) \wedge S'(a_{m'}) \wedge S(a_{m'+1}) \wedge G(a_{m'+2}); 
\label{eq:13:sublem:lem:2:main} 
\\
&\kModel{M},w''\not\models S''(a_{m'}) \vee S'(a_{m'+1}) \vee S(a_{m'+2}) \vee G(a_{m'+3}).
\label{eq:14:sublem:lem:2:main}
\end{flalign}
By \eqref{eq:5:sublem:lem:2:main} and \eqref{eq:12:sublem:lem:2:main}, we readily obtain that
\begin{equation}
\label{eq:15:sublem:lem:2:main} 
\kModel{M},w''\models S'(a_{m-1}) \wedge S(a_{m}) \wedge G(a_{m+1}).
\end{equation}
Then, by $\mathit{Agree}_\lhd$ and \eqref{eq:13:sublem:lem:2:main}--\eqref{eq:15:sublem:lem:2:main}, we can conclude that $m'\geqslant m-1$. Since, by assumption, $m' < m+1$, it follows that
$$
\begin{array}{lcl}
m'=m & \mbox{or} & m'=m-1.
\end{array}
$$
Let us show that each of the cases is impossible. Observe that \eqref{eq:8:sublem:lem:2:main} implies, by $\mathit{Conn}_3$, that
\begin{equation}
\label{eq:16:sublem:lem:2:main} 
\kModel{M},w' \not\models Q(a_{k+1})\wedge S'(a_{m+1})\to Q'(a_{k+1})\vee S''(a_{m+1}).
\end{equation}
If $m'=m$, then, by \eqref{eq:10:sublem:lem:2:main}, \eqref{eq:13:sublem:lem:2:main}, and \eqref{eq:14:sublem:lem:2:main},
$$
\begin{array}{l}
\kModel{M},w''\not\models Q(a_{k+1})\wedge S(a_{m+1})\to Q'(a_{k+1})\vee S'(a_{m+1}),
\end{array}
$$
that, together with \eqref{eq:16:sublem:lem:2:main}, contradicts to $\mathit{Agree}_{S}$. 
If $m'=m-1$, then, again by \eqref{eq:10:sublem:lem:2:main}, \eqref{eq:13:sublem:lem:2:main}, and \eqref{eq:14:sublem:lem:2:main},
$$
\begin{array}{l}
\kModel{M},w''\not\models Q(a_{k+1})\wedge G(a_{m+1})\to Q'(a_{k+1})\vee S(a_{m+1}),
\end{array}
$$
that, together with \eqref{eq:16:sublem:lem:2:main}, contradicts to $\mathit{Agree}_{G}$. Thus, we have proved $\Rght'(k+1)=m+1$, i.e., $\Rght'(k+1)=\Abv(k)+1$. Hence, by \eqref{eq:coonection:right-above},
$$
\begin{array}{lcl}
\Rght'(k+1) & = & \Rght(k+1),
\end{array}
$$
i.e., \eqref{eq:1i:sublem:lem:2:main} is proved.
By $\mathit{Conn}_2$, \eqref{eq:8:sublem:lem:2:main}, $\mathit{Diag}_S$, and \eqref{eq:pair:right'}, we obtain $\Abv'(k+1) = \Rght(k+1)+1$. Hence, by \eqref{eq:coonection:right-above-a},
$$
\begin{array}{lcl}
\Abv'(k+1) & = & \Abv(k+1),
\end{array}
$$
i.e., \eqref{eq:2i:sublem:lem:2:main} is proved.
To prove \eqref{eq:3i:sublem:lem:2:main}, recall that $\Abv(k)=m$, and hence, by the induction hypothesis, $\Wall'(r)=\Wall(r)$ with $0\leqslant r < m$; thus, since $\Abv'(k+1) = \Rght(k+1)+1 = m+2$, we have to prove that $\Wall'(m)=\Wall(m)$ and $\Wall'(m+1)=\Wall(m+1)$. By \eqref{eq:8:sublem:lem:2:main}, we readily obtain $\Wall'(m)$ and $\neg\Wall'(m+1)$. From $\Wall(k)$ we obtain $i_k=0$, and then from $m=\Abv(k)$ we obtain $i_m=0$; thus, $\Wall(m)$. By \eqref{eq:pair:1}, $i_{m+1}\ne 0$; thus, $\neg\Wall(m+1)$. Hence,
$$
\begin{array}{lcl}
\Wall'(r)~=~\Wall(r), & & \mbox{where $0\leqslant r < \Abv(k+1)$,}
\end{array}
$$
i.e., \eqref{eq:3i:sublem:lem:2:main} is proved.

Case $\neg\Wall(k)$:
This case is similar and even simpler; therefore, we give a shorter proof, omitting some details.
By the induction hypothesis, $\neg\Wall'(k)$, which, by~\eqref{eq:pair:wall'}, gives us $\kModel{M},w_0\not\models \wall(a_k)$. Hence, by $\mathit{EM}_W$, $\kModel{M},w\models \neg\wall(a_k)$. From \eqref{eq:4:sublem:lem:2:main}--\eqref{eq:6:sublem:lem:2:main} we obtain
\begin{equation}
\label{eq:7':sublem:lem:2:main} 
\begin{array}{l}
\kModel{M},w\models
\exists y\,(y\lhd a_{k+1}\wedge \wall(y))\wedge \nxt(a_{k+1})\wedge \abv(a_m); 
\smallskip\\
\kModel{M},w\not\models Q(a_{k+1})\vee S'(a_{m}).
\end{array}
\end{equation}
By $\mathit{Move}_2$, \eqref{eq:7':sublem:lem:2:main} implies
$$
\kModel{M},w\not\models \neg\wall(a_m)\wedge \rght(a_m)\wedge Q(a_{k+1}) \to Q'(a_{k+1}) \vee S''(a_m).
$$
Hence, there exists $w'\in R(w)$ such that
\begin{equation}
\label{eq:8':sublem:lem:2:main} 
\begin{array}{l}
\kModel{M},w'\models
\neg\wall(a_m)\wedge \rght(a_{m})\wedge Q(a_{k+1});
\smallskip\\
\kModel{M},w'\not\models Q'(a_{k+1}) \vee S''(a_{m}).
\end{array}
\end{equation}
Since $w_0Rw'$, \eqref{eq:8':sublem:lem:2:main} implies  
\begin{equation}
\label{eq:9':sublem:lem:2:main} 
\kModel{M},w_0\not\models Q(a_{k+1})\wedge \rght\hfill(a_{m})\to Q'(a_{k+1})\vee S''(a_{m}),
\end{equation}
and we conclude, by \eqref{eq:pair:right'}, that $\Rght'(k+1)\leqslant m$.
To prove $\Rght'(k+1) = m$, suppose, for the sake of contradiction, that $\Rght'(k+1) = m' < m$. Then, for some $w''\in R(w_0)$,
\begin{equation}
\label{eq:10':sublem:lem:2:main} 
\begin{array}{lcl}
\kModel{M},w''\models
\rght(a_{m'})\wedge Q(a_{k+1})
  & \mbox{and}
  & \kModel{M},w''\not\models Q'(a_{k+1}) \vee S''(a_{m'}).
\end{array}
\end{equation}
Note that, as in the previous case, 
\begin{equation}
\label{eq:12':sublem:lem:2:main} 
w''\in R(w).
\end{equation}
By $\mathit{Diag}_N$, $\mathit{Diag}_Q$, $\mathit{Diag}_S$, and $\mathit{Conn}_3$, it follows from~\eqref{eq:10':sublem:lem:2:main} that
\begin{flalign}
&\kModel{M},w''\models S''(a_{m'-1}) \wedge S'(a_{m'}) \wedge S(a_{m'+1}); 
\label{eq:13':sublem:lem:2:main} 
\\
&\kModel{M},w''\not\models S''(a_{m'}) \vee S'(a_{m'+1}) \vee S(a_{m'+2}).
\label{eq:14':sublem:lem:2:main}
\end{flalign}
By \eqref{eq:5:sublem:lem:2:main} and \eqref{eq:12':sublem:lem:2:main}, 
\begin{equation}
\label{eq:15':sublem:lem:2:main} 
\kModel{M},w''\models S'(a_{m-1}) \wedge S(a_{m}).
\end{equation}
Then, by $\mathit{Agree}_\lhd$, \eqref{eq:13':sublem:lem:2:main}, and \eqref{eq:15':sublem:lem:2:main}, we can conclude that $m'\geqslant m-1$; since, by assumption, $m' < m$, it follows that $m'=m-1$.
%$$
%\begin{array}{lcl}
%m'=m & \mbox{or} & m'=m-1.
%\end{array}
%$$
Observe that \eqref{eq:8':sublem:lem:2:main} implies, by $\mathit{Conn}_3$, that
\begin{equation}
\label{eq:16':sublem:lem:2:main} 
\kModel{M},w' \not\models Q(a_{k+1})\wedge S'(a_{m})\to Q'(a_{k+1})\vee S''(a_{m}).
\end{equation}
Then, by \eqref{eq:10':sublem:lem:2:main}, \eqref{eq:13':sublem:lem:2:main}, and \eqref{eq:14':sublem:lem:2:main},
$$
\begin{array}{l}
\kModel{M},w''\not\models Q(a_{k+1})\wedge S(a_{m})\to Q'(a_{k+1})\vee S'(a_{m}),
\end{array}
$$
that, together with \eqref{eq:16':sublem:lem:2:main}, contradicts to $\mathit{Agree}_{S}$. 
Thus, $\Rght'(k+1)=m$, i.e., $\Rght'(k+1)=\Abv(k)$. Hence, by \eqref{eq:coonection:right-above},
$$
\begin{array}{lcl}
\Rght'(k+1) & = & \Rght(k+1),
\end{array}
$$
i.e., \eqref{eq:1i:sublem:lem:2:main} is proved.
Using $\mathit{Conn}_2$, \eqref{eq:8':sublem:lem:2:main}, $\mathit{Diag}_S$, and \eqref{eq:pair:right'}, we obtain $\Abv'(k+1) = \Rght(k+1)+1$. Hence, by \eqref{eq:coonection:right-above-a},
$$
\begin{array}{lcl}
\Abv'(k+1) & = & \Abv(k+1),
\end{array}
$$
i.e., \eqref{eq:2i:sublem:lem:2:main} is proved.
To prove \eqref{eq:3i:sublem:lem:2:main}, recall that $\Abv(k)=m$, and hence, by the induction hypothesis, $\Wall'(r)=\Wall(r)$ with $0\leqslant r < m$; thus, since $\Abv'(k+1) = \Rght(k+1)+1 = m+1$, we have to prove that $\Wall'(m)=\Wall(m)$. By \eqref{eq:8':sublem:lem:2:main}, we obtain $\neg\Wall'(m)$. From $\neg\Wall(k)$ we obtain $i_k\ne 0$, and then from $m=\Abv(k)$ we obtain $i_m\ne 0$; thus, $\neg \Wall(m)$. Hence,
$$
\begin{array}{lcl}
\Wall'(r)~=~\Wall(r), & & \mbox{where $0\leqslant r < \Abv(k+1)$,}
\end{array}
$$
i.e., \eqref{eq:3i:sublem:lem:2:main} is proved.
%
%This completes the proof of the sublemma.
\end{proof}

Let us observe that definition~\eqref{eq:lem:2:main:5} can be rewritten in the following way: for every $k\in\numN$ and every $s\in\{0,\ldots,n\}$, 
\begin{equation}
\label{eq:lem:2:main:6}
\begin{array}{lcl}
f(i_k,j_k)=t_s 
  & \bydef 
  & \mbox{$\kModel{M},w_0\models P_s(a_k)$.}
    \phantom{_{\nmbr(i,j)}}
\end{array}
\end{equation}

Let us show that the $T$-tiling defined by \eqref{eq:lem:2:main:6} satisfies~\eqref{eq:T1}. Let $k\in \numN$ and $\kModel{M},w_0\models P_i(a_k)$, for some $i\in\{0,\ldots,n\}$. Let $m=\Rght(k)$ and $\kModel{M},w_0\models P_j(a_m)$. We have to show that $\rightsq t_i = \leftsq t_j$.
Suppose, for the sake of contradiction, that $\rightsq t_i \ne \leftsq t_j$. By Sublemma~\ref{sublem:lem:2:main}, $\Rght'(k)=\Rght(k)$. Therefore, by~\eqref{eq:pair:right'}, 
\begin{equation}
\label{eq:lem:2:main:tiling:T1a}
\kModel{M},w_0\not\models Q(a_k)\wedge \rght(a_m)\to Q'(a_k)\vee S''(a_m).
\end{equation}
Since $\kModel{M},w_0\models P_i(a_k)\wedge P_j(a_m)$, we obtain, by \eqref{eq:lem:2:main:tiling:T1a}, that
\begin{equation}
\label{eq:lem:2:main:tiling:T1b}
\kModel{M},w_0\not\models P_j(a_m) \to \big( Q(a_k)\wedge \rght(a_m) \wedge P_i(a_k) \to Q'(a_k)\vee S''(a_m)\big).
\end{equation}
Since $\rightsq t_i \ne \leftsq t_j$, \eqref{eq:lem:2:main:tiling:T1b} contradicts to~$T_1$. Thus, $\rightsq t_i = \leftsq t_j$, and hence, the $T$-tiling satisfies~\eqref{eq:T1}.

Let us show that the $T$-tiling satisfies~\eqref{eq:T2}. Let $k\in \numN$ and $\kModel{M},w_0\models P_i(a_k)$, for some $i\in\{0,\ldots,n\}$. Let $m=\Abv(k)$ and $\kModel{M},w_0\models P_j(a_m)$. We have to show that $\upsq t_i = \downsq t_j$.
Suppose that $\upsq t_i \ne \downsq t_j$. By Sublemma~\ref{sublem:lem:2:main}, $\Abv'(k)=\Abv(k)$. Therefore, by~\eqref{eq:pair:above'}, 
%\begin{equation}
%\label{eq:lem:2:main:tiling:T2a}
%\kModel{M},w_0\not\models Q(a_k)\wedge \abv(a_m)\to Q'(a_k)\vee S'(a_m).
%\end{equation}
%Since $\kModel{M},w_0\models P_i(a_k)$, we obtain, by \eqref{eq:lem:2:main:tiling:T2a}, that
\begin{equation}
\label{eq:lem:2:main:tiling:T2b}
\kModel{M},w_0\not\models P_j(a_m) \to \big( Q(a_k)\wedge \abv(a_m) \wedge P_i(a_k) \to Q'(a_k)\vee S'(a_m)\big).
\end{equation}
Since $\upsq t_i \ne \downsq t_j$, \eqref{eq:lem:2:main:tiling:T2b} contradicts to~$T_2$. Thus, $\upsq t_i = \downsq t_j$, and hence, the $T$-tiling satisfies~\eqref{eq:T2}.
\end{proof}

\section{Immediate theorems}
\setcounter{equation}{0}

Let us infer some corollaries of the construction described above. First of all, now we are ready to answer the question about the decidability of the two-variable fragment of~$\logic{QLC}$.

\begin{theorem}
\label{th:QLC}
Logic\/ $\logic{QLC}$ is\/ $\Sigma^0_1$-complete in the language with two individual variables, a single binary predicate letter, and infinitely many unary predicate letters.
\end{theorem}

\begin{proof}
Immediately follows from Lemmas~\ref{lem:1:main} and~\ref{lem:2:main}.
\end{proof}

Next, let us observe that we can expand the construction on some extentions of $\logic{QLC}$, in particular, defined by classes of linear c\nobreakdash-augmented frames. 

\begin{theorem}
\label{th:QLC:nat-c}
Let a logic $L$ be such that $\logic{QLC}\subseteq L\subseteq \QSIL(\kframe{G}\odot\numN)$. Then $L$ is\/ $\Sigma^0_1$-hard in the language with two individual variables, a single binary predicate letter, and infinitely many unary predicate letters.
\end{theorem}

\begin{proof}
Follows from Lemmas~\ref{lem:1:main} and~\ref{lem:2:main}. Indeed, let $T$ be a finite set of tile types. If there is no $T$-tiling satisfying~\eqref{eq:T1} and~\eqref{eq:T2}, then, by Lemma~\ref{lem:2:main}, $\varphi_T$ is not refuted on the class of Kripke frames validating $\logic{QLC}$. Then $\varphi_T\in \logic{QLC}$, and hence, $\varphi_T\in L$. If there exists a $T$-tiling satisfying~\eqref{eq:T1} and~\eqref{eq:T2}, then, by Lemma~\ref{lem:1:main}, $\kframe{G}\odot\numN\not\models\varphi_T$. Thus, $\varphi_T\not\in \QSIL(\kframe{G}\odot\numN)$, and hence, $\varphi_T\not\in L$.
\end{proof}

\begin{corollary}
Logic\/ $\logic{QLC.cd}$ is\/ $\Sigma^0_1$-complete in the language with two individual variables, a single binary predicate letter, and infinitely many unary predicate letters.
\end{corollary}

\begin{corollary}
Let\/ $\kframe{F}$ be one of the Kripke frames\/ $\otuple{\numZ,\leqslant}$, $\otuple{\numR,\leqslant}$\/ or\/ $\otuple{\alpha,{\subseteq}}$, where\/ $\alpha$\/ is an infinite ordinal. Then\/ $\QSILext{e}{}\kframe{F}$\/ and\/ $\QSILext{c}{}\kframe{F}$ are\/ $\Sigma^0_1$-hard in the language with two individual variables, a single binary predicate letter, and infinitely many unary predicate letters.
\end{corollary}

\section{High undecidability}
\label{sec:hight:undec}
\setcounter{equation}{0}

%Let $\kframe{G}=\otuple{\numN,\leqslant}$.
Here, we show that the two-variable fragments of all logics between $\QSILext{e}{}\kframe{G}$ and $\QSIL(\kframe{G}\odot\numN)$ are also $\Pi^0_1$-hard. To that end, let us consider a different tiling problem.
Let $f\colon\numN\times\numN\to T$ be a $T$-tiling, where $T=\{t_0,t_1,\ldots,t_n\}$; we define two else~--- in addition to \eqref{eq:T1} and \eqref{eq:T2}~--- conditions for it. The first~is 
%The first condition is that the edge colors of the adjacent tiles match horizontally: for all $i, j \in \numN$,
\begin{equation}
\label{eq:T3}
\begin{array}{lcl}
\phantom{j^\ast+j}
f(0,0) & = & t_0;
\phantom{00}
\end{array}
\end{equation}
the second is that there exists $j^\ast\in \numN^+$ such that, for every $j\in\numN$,
\begin{equation}
\label{eq:T4}
\begin{array}{lcl}
\phantom{0}
f(0,j^\ast+j) & = & t_1.
\phantom{00}
\end{array}
\end{equation}
The tiling problem we consider is the following: given a set $T=\{t_0,t_1,\ldots,t_n\}$ of tile types, we are to determine whether there exists a $T$-tiling satisfying~\eqref{eq:T1}, \eqref{eq:T2}, \eqref{eq:T3}, and~\eqref{eq:T4}. It can be shown that this tiling problem is $\Sigma^0_1$-hard, see Section~\ref{sec:app}. 
We are going to simulate this tiling problem with formulas of the same fragment.

Let us introduce the following abbreviation:
%$$
\begin{equation}
\label{eq:preceq}
\begin{array}{rcl}
x\preccurlyeq y 
  & = 
  & Q(y)\to Q(x). 
  \smallskip\\
\end{array}
\end{equation}
%$$

Then, using it, let us introduce formulas, which, as we shall show, allow us to simulate the tiling problem just defined: 
$$
\begin{array}{rcl}
\mathit{Agree}_\preccurlyeq & = & \forall x\forall y\,(x\lhd y \to x\preccurlyeq y);
\smallskip\\
%\end{array}
%$$
%Next, let us define formulas describing conditions \eqref{eq:T3} and~\eqref{eq:T4}:
%$$
%\begin{array}{lcl}
T_3 & = & \forall x\,\big(\wall(x)\wedge\floor(x)\to P_0(x)\big);
\smallskip\\
T_4 & = & \exists x\forall y\,\big(x\preccurlyeq y \wedge \wall(y)\to P_1(y)\big);
\smallskip\\
%\end{array}
%$$
%Finally, define
%$$
%\begin{array}{lcl}
\mathit{Refute}_Q & = & \exists x\,(Q(x)\to Q'(x));
\smallskip\\
\mathit{Grid}' & = & \mathit{Grid} \wedge \mathit{Agree}_\preccurlyeq;
\smallskip\\
\mathit{Tiling}'_T & = & \mathit{Tiling}_T\wedge T_3\wedge T_4;
\smallskip\\
\mathit{Refute}' & = & \mathit{Refute}\vee \mathit{Refute}_Q;
\smallskip\\
%\end{array}
%$$
%and then put
%$$
%\begin{array}{lcl}
\psi_T & = & \mathit{Grid}'\wedge \mathit{Tiling}'_T\to \mathit{Refute}'.
\smallskip\\
\end{array}
$$

\begin{lemma}
\label{lem:1:main:nat}
If there exists a\/ $T$-tiling satisfying\/~\eqref{eq:T1}, \eqref{eq:T2}, \eqref{eq:T3},\/ and\/ \eqref{eq:T4}, then\/ $\otuple{\numN,\leqslant}\odot\numN,0\not\models\psi_T$.
\end{lemma}

\begin{proof}
Just follow the proof of Lemma~\ref{lem:1:main}. Note along the way that if the\/ $T$-tiling $f\colon\numN\times\numN\to T$ satisfies \eqref{eq:T1}, \eqref{eq:T2}, \eqref{eq:T3},\/ and\/ \eqref{eq:T4}, then, for all $k,m\in\numN$,
$$
\begin{array}{lcl}
\kModel{M},0\models k\preccurlyeq m & \iff & k\leqslant m,
\end{array}
$$
where $\kModel{M}$ is a model defined as in the proof of Lemma~\ref{lem:1:main}.
\end{proof}

\begin{lemma}
\label{lem:2:main:nat}
Let\/ $\kframe{G}\not\models\psi_T$.
Then there exists a $T$-tiling satisfying\/~\eqref{eq:T1} \eqref{eq:T2}, \eqref{eq:T3},\/ and\/ \eqref{eq:T4}.
\end{lemma}

\begin{proof}
Let us follow the proof of Lemma~\ref{lem:2:main}.
Let $\kModel{M}=\otuple{\kframe{G},D,I}$ and $w^\ast$ be, respectively, a model over $\kframe{G}$ and a world of $\kframe{G}$ such that
$$
\begin{array}{lcl}
\kModel{M},w^\ast \models \mathit{Grid}'\wedge \mathit{Tiling}'_T 
  & \mbox{and}
  & \kModel{M},w^\ast\not\models \mathit{Refute}'.
\end{array}
$$
Then, clearly, in particular,
$$
\begin{array}{lcl}
\kModel{M},w^\ast \models \mathit{Grid}\wedge \mathit{Tiling}_T
  & \mbox{and}
  & \kModel{M},w^\ast\not\models \mathit{Refute},
\end{array}
$$
that corresponds to~\eqref{eq:lem:2:main:0}; hence, we may assume all the constructions in the proof of Lemma~\ref{lem:2:main} being done. In particular, the $T$-tiling $f\colon\numN\times\numN\to T$ defined by~\eqref{eq:lem:2:main:5} satisfies \eqref{eq:T1} and~\eqref{eq:T2}, and we have to show that it also satisfies \eqref{eq:T3} and~\eqref{eq:T4}. 

Observe that it satisfies \eqref{eq:T3} by \eqref{eq:lem:2:main:5} and~$T_3$.

Let us show that it satisfies \eqref{eq:T4} as well.
To that end, let us choose, for every $k\in\numN^+$, a world $w_k$ in $\kframe{G}$ satisfying the condition
\begin{equation}
\label{eq:lem:2:main:nat:1}
\begin{array}{lcl}
\kModel{M},w_k \models Q(a_k) 
  & \mbox{and}
  & \kModel{M},w_k\not\models Q'(a_k);
\end{array}
\end{equation}
observe that $w_0$ defined in the proof of Lemma~\ref{lem:2:main} satisfies \eqref{eq:lem:2:main:nat:1}, too (i.e., with $k=0$). 
It easily follows also, by $\mathit{Agree}_\preccurlyeq$ and $\mathit{Diag}_Q$, that, for all $k,m\in\numN$, 
\begin{equation}
\label{eq:lem:2:main:nat:2}
\begin{array}{lcl}
k < m
  & \imply 
  & \mbox{both $\kModel{M},w_0 \models a_k \preccurlyeq a_m$ and $w_k < w_m$.} 
\end{array}
\end{equation}
%with $w_k\ne w_m$ whenever $k\ne m$.
By $T_4$, there exists $b\in D_{w_0}$ such that, for every $a\in D_{w_0}$, 
\begin{equation}
\label{eq:lem:2:main:nat:3}
\begin{array}{lcl}
\kModel{M},w_0 \models b\preccurlyeq a \wedge \wall(a)
  & \imply
  & \kModel{M},w_0 \models P_1(a).
\end{array}
\end{equation}
Since $\kModel{M},w^\ast\not\models \mathit{Refute}_Q$, there exists $w\in\numN$ such that 
$$
%\begin{equation}
%\label{eq:lem:2:main:nat:4}
\begin{array}{lcl}
\kModel{M},w \models Q(b) 
  & \mbox{and}
  & \kModel{M},w\not\models Q'(b).
\end{array}
%\end{equation}
$$
By \eqref{eq:lem:2:main:nat:2}, the set $\{w_k : k\in\numN\}$ is infinite, therefore, there exists $k\in\numN$ such that $w_{k}\in R(w)$. It should be clear that 
\begin{equation}
\label{eq:lem:2:main:nat:5}
\kModel{M},w_0 \models b\preccurlyeq a_k. 
\end{equation}
Then, for every $m\in \numN^+$ with $k \leqslant m$, it follows by \eqref{eq:lem:2:main:nat:2} that $\kModel{M},w_0 \models a_k \preccurlyeq a_m$. 
Observe that $\preccurlyeq$ defines a linear preorder on $D_{w_0}$; hence, \eqref{eq:lem:2:main:nat:5} implies $\kModel{M},w_0 \models b \preccurlyeq a_m$.
If $\Wall(m)$, then, by \eqref{eq:lem:2:main:nat:3}, $\kModel{M},w_0 \models P_1(a_m)$, i.e., \eqref{eq:T4} is satisfied with $j^\ast=k$.  
\end{proof}

As an immediate corollary, we obtain the following statement.

\begin{theorem}
\label{th:main:nat}
Let\/ $\QSILext{e}{}\kframe{G}\subseteq L \subseteq \QSIL(\kframe{G}\odot\numN)$. Then $L$ is both\/ $\Sigma^0_1$-hard and\/ $\Pi^0_1$-hard in the language with two individual variables, a single binary predicate letter, and infinitely many unary predicate letters.
\end{theorem}

\begin{proof}
Immediately follows from Theorem~\ref{th:QLC:nat-c}, Lemma~\ref{lem:1:main:nat}, and Lemma~\ref{lem:2:main:nat}.
\end{proof}

\Rem{%%%%%%%%%%%%%%%%%%%%%%%%%%%%%%%%%%%%%%%%%%%%%
Clearly, $\preccurlyeq$ defines a reflexive and transitive relation, i.e., a preorder.

Notice that $\forall x\forall y\,((Q(x)\to Q(y)) \vee (Q(y)\to Q(x)))\in \logic{QLC}$; hence, in domains of worlds of augmented $\logic{QLC}$-frames, $\preccurlyeq$ corresponds to (non-strict) linear preorders.

============

First of all, let us observe that $\preccurlyeq^{I,w}$ is a linear preorder on $D_w$, for every $w\in W$. Indeed, to see that $\preccurlyeq^{I,w}$ is reflexive, just observe that $\kModel{M},w\models Q(a)\to Q(a)$, for every $a\in D_w$. Next, $\kModel{M},w\models a\preccurlyeq b \wedge b\preccurlyeq c \to a\preccurlyeq c$, for all $a,b,c\in D_w$, since $(\varphi\to \varphi')\wedge (\varphi'\to \varphi'') \to (\varphi\to \varphi'')\in\logic{QInt}$, for all formulas $\varphi$, $\varphi'$, and~$\varphi''$; thus, $\preccurlyeq^{I,w}$ is transitive. 
%Hence, $\preccurlyeq^{I,w}$ is a preorder. 
Finally, $\kframe{F}\models\logic{QLC}$ implies that $\kModel{M},w\models (Q(a)\to Q(b))\vee (Q(b)\to Q(a))$, for all $a,b\in D_w$; by the definition of~$\preccurlyeq$, it means that $\kModel{M},w\models a\preccurlyeq b \vee b\preccurlyeq a$. 
%Thus, $\preccurlyeq^{I,w}$ is a linear preorder.

============

Then, by $\mathit{Inclusion}^\preccurlyeq_\lhd$ and \eqref{eq:lem:2:main:3}, for all $k,m\in\numN$,
\begin{equation}
\label{eq:lem:2:main:4}
\begin{array}{lcl}
k\leqslant m & \imply & \kModel{M},w_0\models a_k\preccurlyeq a_m,
\end{array}
\end{equation}
that can be proved by induction on~$m$. 

============
}%%%%%%%%%%%%%%%%%%%%%%%%%%%%%%%%%%%%%%%%%

\section{Positive fragments}
\setcounter{equation}{0}

A formula $\varphi$ is \defnotion{positive} if $\varphi$ does not contain occurrences of~$\bot$.
For a superintuitionistic logic~$L$, let $L^+$ denote the \defnotion{positive fragment} of~$L$, i.e., the subset of~$L$ consisting of the positive formulas.

Our next step is to show that Theorem~\ref{th:QLC:nat-c} remains true for the positive fragments of the logics. 
To that end, let us eliminate $\bot$ in~$\varphi_T$ and~$\psi_T$. Let $\varphi^+_T$ and~$\psi^+_T$ be the formulas obtained from $\varphi_T$ and~$\psi_T$, respectively, by replacing every occurrence of $\bot$ in $\varphi_T$ and~$\psi_T$ with $\forall x\,Q'(x)$.

\begin{lemma}
\label{lem:1:main:positive}
If there exists a\/ $T$-tiling satisfying\/~\eqref{eq:T1}\/ and\/ \eqref{eq:T2}, then\/ $\kframe{G}\odot\numN,0\not\models\varphi^+_T$.
\end{lemma}

\begin{proof}
Let $\kModel{M}$ be the model constructed in the proof of Lemma~\ref{lem:1:main}. Note that $\kModel{M}$ is based on $\kframe{G}\odot\numN$ and observe that $\kModel{M}\models \forall x\,Q'(x) \leftrightarrow \bot$. Therefore, $\kModel{M},0\not\models\varphi^+_T$, and hence, $\kframe{G}\odot\numN,0\not\models\varphi^+_T$.
\end{proof}

\begin{lemma}
\label{lem:2:main:positive}
Let\/ $\kframe{F}=\otuple{W,R}$ be a Kripke frame validating\/ $\logic{QLC}$ such that $\kframe{F}\not\models\varphi^+_T$.
Then there exists a $T$-tiling satisfying conditions\/~\eqref{eq:T1}~and\/~\eqref{eq:T2}.
\end{lemma}

\begin{proof}
Just follow, step by step, the proof of Lemma~\ref{lem:2:main}. Observe that every use of the fact of the form $\kModel{M},w\models \neg\psi$ (i.e., $\kModel{M},w\models \psi\to\bot$) in the proof can be replaced with the use of $\kModel{M},w\models \psi\to\forall x\,Q'(x)$, since $\kModel{M},w\not\models Q'(a_k)$, for some $k\in\numN$, also.
\end{proof}

These observations provide us with the following theorem stating the undecidability for the positive fragments of the logics, which is a refinement of Theorem~\ref{th:QLC:nat-c}.

\begin{theorem}
\label{th:QLC:nat-c:positive}
Let a logic $L$ be such that $\logic{QLC}\subseteq L\subseteq \QSIL(\kframe{G}\odot\numN)$. Then $L^+$ is\/ $\Sigma^0_1$-hard in the language with two individual variables, a single binary predicate letter, and infinitely many unary predicate letters.
\end{theorem}

\begin{proof}
Similar to the proof of Theorem~\ref{th:QLC:nat-c:positive}.
Follows from Lemmas~\ref{lem:1:main:positive} and~\ref{lem:2:main:positive}.
\end{proof}

As a corollary of Theorem~\ref{th:QLC:nat-c:positive} we immediately obtain the following statement.

\begin{corollary}
\label{cor:QLC:positive}
The positive fragments of both\/ $\logic{QLC}$\/ and\/ $\logic{QLC.cd}$\/ are\/ $\Sigma^0_1$-complete, even in the language with two individual variables, a single binary predicate letter, and infinitely many unary predicate letters.
\end{corollary}

Finally, let us state a corollary for the positive fragments of some other special extensions of $\logic{QLC}$ mentioned above.

\begin{corollary}
Let\/ $\kframe{F}$ be one of the Kripke frames\/ $\otuple{\numZ,\leqslant}$, $\otuple{\numR,\leqslant}$\/ or\/ $\otuple{\alpha,{\subseteq}}$, where\/ $\alpha$\/ is an infinite ordinal. Then\/ $(\QSILext{e}{}\kframe{F})^+$\/ and\/ $(\QSILext{c}{}\kframe{F})^+$ are\/ $\Sigma^0_1$-hard in the language with two individual variables, a single binary predicate letter, and infinitely many unary predicate letters.
\end{corollary}

Of course, similar statements hold for the logics between $\QSILext{e}{}\kframe{G}$ and $\QSIL(\kframe{G}\odot\numN)$ but concerning $\Pi^0_1$-hardness.

\begin{lemma}
\label{lem:1:main:nat:positive}
If there exists a\/ $T$-tiling satisfying\/~\eqref{eq:T1}, \eqref{eq:T2}, \eqref{eq:T3}\/ and\/ \eqref{eq:T4}, then\/ $\kframe{G}\odot\numN,0\not\models\psi^+_T$.
\end{lemma}

\begin{proof}
Similar to the proof of Lemma~\ref{lem:1:main:positive}.
\end{proof}

\begin{lemma}
\label{lem:2:main:nat:positive}
Let\/ $\kframe{G}\not\models\psi^+_T$.
Then there exists a $T$-tiling satisfying\/~\eqref{eq:T1} \eqref{eq:T2}, \eqref{eq:T3},\/ and\/ \eqref{eq:T4}.
\end{lemma}

\begin{proof}
Similar to the proof of Lemma~\ref{lem:2:main:positive}.
\end{proof}

\begin{theorem}
\label{th:main:nat:pos}
Let\/ $\QSILext{e}{}\kframe{G}\subseteq L \subseteq \QSIL(\kframe{G}\odot\numN)$. Then $L^+$ is both\/ $\Sigma^0_1$-hard and\/ $\Pi^0_1$-hard in the language with two individual variables, a single binary predicate letter, and infinitely many unary predicate letters.
\end{theorem}

\begin{proof}
Immediately follows from Theorem~\ref{th:QLC:nat-c:positive}, Lemma~\ref{lem:1:main:nat:positive}, and Lemma~\ref{lem:2:main:nat:positive}.
\end{proof}

\section{Conclusion}
\setcounter{equation}{0}

The undecidability of $\logic{QLC}$ is proved, and, clearly, the methods used can be applied to other logics. In particular, they can be used in various modal predicate logics, because the modal language is more flexible than intuitionistic. Some results in this direction have already been obtained, and perhaps they will be published soon.
% in a joint work with D.\,Shkatov. 
Also, the methods seem working in the predicate counterparts of the basic and formal propositional logics introduced by A.\,Visser~\cite{Visser81}.

Let us turn to the language used. Two questions naturally arise: Do we need a binary predicate letter in the construction? Do we need infinitely many unary predicate letters in it?

It is known that in many cases the atomic predicate formulas can be simulated by the monadic~--- modal or superintuitionistic~--- formulas~\cite{Kripke62,HC96,MR:2001:LI,MR:2002:LI,KKZ05,RSh19SL,RShJLC20a,RShJLC21b}. So, $P(x,y)$ can be simulated by $\Diamond(Q_1(x)\wedge Q_2(y))$ in the modal language~\cite{Kripke62,HC96} and by $(Q_1(x)\wedge Q_2(y)\to p)\vee q$ in the intuitionistic~\cite{RShJLC21b,RShsubmitted2}.\footnote{In the intuitionistic case, the method can be applied to the positive formulas only; however, $\bot$ can be eliminated in `natural' situations~\cite{RShsubmitted2}.} Notice that there are limitations that do not allow the use of such simulations in some cases: in Kripke frames, we need worlds seeing `sufficiently many' worlds (in the modal case) or seeing `sufficiently large' antichains of worlds (in the intuitionistic case). In view of the last requirement, such simulation is not applicable to~$\logic{QLC}$ and its extensions. 

However, the binary predicate letter we used corresponds in models, as we have seen, to a quite special relation which can be understood as `taking the next element' in a linearly ordered set: so, in the proof of Lemma~\ref{lem:1:main} we define $k\lhd m$ as $m=k+1$ and in the proof of Lemma~\ref{lem:2:main} the truth of $a_k \lhd a_m$ means again that $m=k+1$. There is no difficulty to define a linear order that agrees with such `next' relation: just take $\preccurlyeq$ defined by~\eqref{eq:preceq} and use $\mathit{Agree}_\preccurlyeq$.
%
%
%just put, for individuals $a$ and $b$,
%%$$
%\begin{equation}
%\label{eq:prec}
%\begin{array}{lcl}
%a \preccurlyeq b & \bydef & Q(b)\to Q(a),
%\end{array}
%\end{equation}
%%$$
%and then add the condition $\forall x\forall y\,(x\lhd y\to (Q(y)\to Q(x)))$. 
Clearly, in a model based on a frame validating~$\logic{QLC}$, relations satisfying these conditions are linear preorders. But then we run into a problem: is it possible to define something like $x\lhd y$ using the binary relation corresponding to $Q(y)\to Q(x)$? Of course, this is not hard to do if there is a third individual variable. Moreover, sometimes this problem can be easily solved in the modal case by means of the monadic language with two individual variables~\cite{MR:2024:IGPL}. At the same time, there are results showing us that the monadic fragments of modal predicate logics can be decidable~\cite{MR:2017:LI,RShsubmitted,RShsubmitted2}; they are obtained for logics complete with respect to quite simple semantics, and the semantics for~$\logic{QLC}$ does not seem to be too complicated. As a summary, the answer to the first question is unclear to the author.

The situation with the second question is quite similar. On the one hand, to simulate all unary predicate letters in an intuitionistic formula by formulas with a single one~\cite{RSh19SL,RShJLC21b}, the construction~\cite{Rybakov06,Rybakov08} used requires frames containing `quite large' antichains. On the other hand, there are constructions for modal logics that allow us to simulate all unary predicate letters in a modal formula by formulas with a single unary letter, and, which is important, to deal with linear Kripke frames only~\cite{RSh20AiML,RShJLC21c,MR:2024:IGPL}. So, the answer to the second question is unclear to the author, too.

Nevertheless, despite the non-obviousness, these questions do not seem too hopeless, and perhaps they will be the subject of future research by the author.

Another question concerns the logic of $\omega^\ast$ viewed as a Kripke frame; and the same holds for logics of linear Noetherian orders as well. It seems that we can argue similarly to Section~\ref{sec:hight:undec}. Indeed, we can define a linear preorder $\preccurlyeq$~by
$$
\begin{array}{lcl}
x\preccurlyeq y & \bydef & Q(x)\to Q(y),
\end{array}
$$
and then agree it with $\lhd$ as before. Actually, it is not hard to prove $\Pi^0_1$-hardness of the logics but with three individual variables in the language; also, it is known that in the modal case, two variables are sufficient even for $\Pi^1_1$-hardness~\cite{MR:2024:IGPL}. The algorithmic complexity of the superintuitionistic logics of Noetherian orders in the language with two individual variables is, again, the subject of future research.

\section{Appendix}
\label{sec:app}
\setcounter{equation}{0}

Here, we consider the following tiling problem: given a set $T=\{t_0,t_1,\ldots,t_n\}$ of tile types, we are to determine whether there exists a $T$-tiling satisfying~\eqref{eq:T1}, \eqref{eq:T2}, \eqref{eq:T3}, and~\eqref{eq:T4}. We shall show that this problem is $\Sigma^0_1$-hard. 

The idea is to encode Turing machines with tile types so that the $k$th row in a tiling corresponds to the $k$th configuration of the computation when a machine starts with the blank tape; when encoding, we add to the program of the machine a fictitious instruction impose the computation to loop in the halting state. Since we may assume Turing machines to stop at the leftmost cell of the tape only, we then obtain that the leftmost tile of the tiling starts repeating from row to row. The rest follows from $\Sigma^0_1$-completeness of the corresponding halting problem for Turing machines.

In the rest of the section, we just describe the construction in detail.

Let us briefly define the modification of Turing machines~\cite{LP98,Sipser12} we shall use. 
%We consider single-tape deterministic Turing machines. Their special features are a finite set of halting states instead of one halting state and instructions beginning with halting states. These instructions work infinitely when we reach a halting state; they are useful to make tiling we shall consider infinite.
A \defnotion{Turing machine} is a tuple $M = \langle \Sigma, Q, q_0, q_1, \delta \rangle$, where $\Sigma$ is a finite alphabet such that $\Box, \# \in \Sigma$ ($\Box$ is the \defnotion{blank symbol} and $\#$ is the \defnotion{end tape marker symbol}), $Q$ is a finite set of \defnotion{states}; $q_0 \in Q$ is the \defnotion{initial state}; $q_1 \in Q$ is the \defnotion{halting state}; and $\delta$ is a \defnotion{program}, i.e., a function $\delta\colon Q\times\Sigma\to Q\times\Sigma\times\{L,S,R\}$ satisfying the following conditions: if $\delta\colon{q}{s}\mapsto{q'}{s'}{\Delta}$, then 
\begin{itemize}
\item $s=\#$ if, and only if, $s'=\#$; 
\item $\Delta\ne L$ whenever $s=\#$; 
\item $q'=q_1$, $s'=s$, and $\Delta=S$ whenever $q=q_1$.
\end{itemize}
The last condition is, in fact, the only modification we need.

\pagebreak[3]

A \defnotion{configuration} of a machine $M = \langle \Sigma, Q, q_0, q_1, \delta \rangle$ is an $\omega$-word $vqv'$, where $q \in Q$ and $vv'=a_0a_1a_2\ldots{}$ is an $\omega$-word over $\Sigma$ satisfying the following conditions: 
\begin{itemize}
\item there exists $k\in\numNp$ such that $a_i = \Box$, for every $i\geqslant k$;
\item $a_i = \#$ if, and only if, $i=0$.
\end{itemize}

A Turing machine $M$ can be thought of as a computing device equipped with a tape divided into an infinite sequence of cells $c_0,c_1,c_2,\ldots{}$, each containing a symbol from $\Sigma$, with one cell being scanned by a movable head. Then, a configuration $vqv'$ of $M$ represents a computation instant at which the tape contains the symbols of the word $vv'$, $M$ is in state $q$, and the head is scanning the cell containing the first symbol of~$v'$.  An \defnotion{instruction} $\delta\colon{q}{s}\mapsto{q'}{s'}{\Delta}$ is applicable to this configuration just in case $M$ is in state $q$ and is scanning a cell containing~$s$.  As a result of applying this instruction, $M$ enters state $q'$, replaces $s$ with $s'$ in the cell, and either moves one cell to the left or to the right or stays put, depending on whether $\Delta$ is $L$, $R$ or~$S$, respectively. Given a word $x$ over $\Sigma\setminus\{\Box,\#\}$ as an \defnotion{input}, $M$ consecutively executes the instructions of~$\delta$ starting from the configuration $q_0\# x \Box\Box\Box\ldots{}$; if $M$ reaches a configuration whose state component $q$ is $q_1$, then $M$ \defnotion{halts} on $x$.
%, which is denoted by $!M(x)$; if $M$ does not halt on $x$, we write $\neg !M(x)$. 
Without a loss of generality, we may assume that the cell being scanned when $M$ halts is $c_0$ (which contains~$\#$). Notice that then the instruction $\delta\colon q_1\#\mapsto {q_1\#}S$ can be applied providing us with the same configuration. This means that even if $M$ halts, we may consider the infinite computation of $M$, in which $M$ loops the same halting configuration.

\Rem{%%%%%%%%%%%%%%%%%%%%%%%%%%%%
\begin{figure}
\centering
\begin{tikzpicture}[scale=2.1]
\drawtileflattm{(0,0)}{(1,0)}{(1,1)}{(0,1)}{$\leftsq t$}{$\downsq t$}{$\rightsq t$}{$\upsq t$}{$t$}
\end{tikzpicture}
\caption{Tile with marks}
\label{fig:1}
\end{figure}
}%%%%%%%%%%%%%%%%%%%%%%%%%%%%%%%%
\begin{figure}
\centering
\begin{tikzpicture}[scale=2.1]
\drawtileflattm{(0,0)}{(1,0)}{(1,1)}{(0,1)}{{$\otimes$}}{{$\otimes$}}{{${\ast}{\ast}$}}{{$q_0\#$}}{$t_0$}
\end{tikzpicture}
\hspace{0.75em}
\begin{tikzpicture}[scale=2.1]
\drawtileflattm{(0,0)}{(1,0)}{(1,1)}{(0,1)}{${\ast}{\ast}$}{$\otimes$}{${\ast}{\ast}$}{$\Box$}{$t_\Box^{\ast\ast}$}
\end{tikzpicture}
\hspace{0.75em}
\begin{tikzpicture}[scale=2.1]
\drawtileflattm{(0,0)}{(1,0)}{(1,1)}{(0,1)}{{$\ast$}}{{$s$}}{{$\ast$}}{{$s$}}{$t_s^\ast$}
\end{tikzpicture}
\hspace{0.75em}
\begin{tikzpicture}[scale=2.1]
\drawtileflattm{(0,0)}{(1,0)}{(1,1)}{(0,1)}{{$\otimes$}}{{$\#$}}{{$\ast$}}{{$\#$}}{$t_{\#}^\ast$}
\end{tikzpicture}
\caption{Tile types $t_0$, $t_\Box^{\ast\ast}$, $t_s^\ast$, and $t_{\#}^\ast$}
\label{fig:3}
\end{figure}

Let us encode a Turing machine~$M$ with a set of tile types.
%; we will accompany the description with drawings of tiles; see  Figure~\ref{fig:1}.
First of all, let us define tile types $t_0$ and $\leftsq  t_\Box^{\ast\ast}$, which are the same for all Turing machines:
\begin{itemize}
\item $\leftsq t_0 = \otimes$, $\rightsq t_0 = {\ast\ast}$, $\upsq t_0 = q_0\#$, $\downsq t_0 = \otimes$;
\item 
$\leftsq  t_\Box^{\ast\ast} = \rightsq t_\Box^{\ast\ast} = {\ast}{\ast}$, 
$\upsq    t_\Box^{\ast\ast} = \Box$, 
$\downsq  t_\Box^{\ast\ast} = \otimes$,
\end{itemize}
where $\otimes$ and $\ast$ are new symbols. Next, for every $s\in\Sigma\setminus\{\#\}$, define $t_s^\ast$ and $t_{\#}^\ast$ by
\begin{itemize}
\item $\leftsq t_s^\ast = \rightsq t_s^\ast = \ast$, $\upsq t_s^\ast = \downsq t_s^\ast = s$;
\item $\leftsq t_{\#}^\ast = \otimes$, $\rightsq t_{\#}^\ast = \ast$, $\upsq t_{\#}^\ast = \downsq t_{\#}^\ast = \#$,
\end{itemize}
see Figure~\ref{fig:3}. 
For every $q\in Q_0$ and $s\in\Sigma_0$, we define tile types depending on the instruction $\delta\colon qs\mapsto q's'\Delta$. 
%There are three cases. 
%
%First: 
If $\delta_0\colon qs\mapsto q's'S$, then define $t_{qs}$ by
\begin{itemize}
\item $\leftsq t_{qs} = \rightsq t_{qs} = \ast$, $\upsq t_{qs} = q's'$, $\downsq t_{qs} = qs$, where $s\ne\#$;
\item $\leftsq t_{q\#} = \otimes$, $\rightsq t_{q\#} = \ast$, $\upsq t_{q\#} = q'\#$, $\downsq t_{q\#} = q\#$,
\end{itemize}
see Figure~\ref{fig:4}. 
\begin{figure}
\centering
\begin{tikzpicture}[scale=2.1]
\drawtileflattm{(0,0)}{(1,0)}{(1,1)}{(0,1)}{{$\ast$}}{{$qs$}}{{$\ast$}}{{$q's'$}}{$t_{qs}$}
\end{tikzpicture}
\hspace{1em}
\begin{tikzpicture}[scale=2.1]
\drawtileflattm{(0,0)}{(1,0)}{(1,1)}{(0,1)}{{$\otimes$}}{{$q\#$}}{{$\ast$}}{{$q'\#$}}{$t_{q\#}$}
\end{tikzpicture}
\caption{Tile types for the instructions $\delta_0\colon qs\mapsto q's'S$ and $\delta_0\colon q\#\mapsto q'\#S$}
\label{fig:4}
\end{figure}
%
%Second: 
If $\delta_0\colon qs\mapsto q's'R$, then define $t_{qs}$ and $t_{qs}^{a}$, for every $a\in\Sigma_0\setminus\{\#\}$, by
\begin{itemize}
\item 
$\leftsq t_{qs} = \ast$,
$\rightsq t_{qs} = qs$, 
$\upsq t_{qs} = s'$, 
$\downsq t_{qs} = qs$, where $s\ne\#$;
\item 
$\leftsq t_{q\#} = \otimes$,
$\rightsq t_{q\#} = q\#$, 
$\upsq t_{q\#} = \#$, 
$\downsq t_{q\#} = q\#$;
\item 
$\leftsq t_{qs}^{a} = qs$,
$\rightsq t_{qs}^{a} = \ast$, 
$\upsq t_{qs}^{a} = q'a$, 
$\downsq t_{qs}^{a} = a$,
\end{itemize}
see Figure~\ref{fig:5}. 
\begin{figure}
\centering
\begin{tikzpicture}[scale=2.1]
\drawtileflattm{(0,0)}{(1,0)}{(1,1)}{(0,1)}{{$\ast$}}{{$qs$}}{{$qs$}}{{$s'$}}{$t_{qs}$}
\drawtileflattm{(1,0)}{(2,0)}{(2,1)}{(1,1)}{{$qs$}}{{$a$}}{{$\ast$}}{{$q'a$}}{$t_{qs}^{a}$}
\end{tikzpicture}
\hspace{1em}
\begin{tikzpicture}[scale=2.1]
\drawtileflattm{(0,0)}{(1,0)}{(1,1)}{(0,1)}{{$\otimes$}}{{$q\#$}}{{$q\#$}}{{$\#$}}{$t_{q\#}$}
\drawtileflattm{(1,0)}{(2,0)}{(2,1)}{(1,1)}{{$q\#$}}{{$a$}}{{$\ast$}}{{$q'a$}}{$t_{q\#}^{a}$}
\end{tikzpicture}
\caption{Tile types for the instructions $\delta_0\colon qs\mapsto q's'R$ and $\delta_0\colon q\#\mapsto q'\#R$}
\label{fig:5}
\end{figure}
%
%Third: 
If $\delta_0\colon qs\mapsto q's'L$, then define $t_{qs}$ and $t_{qs}^{a}$, for every $a\in\Sigma_0$, by
\begin{itemize}
\item 
$\leftsq t_{qs} = qs$,
$\rightsq t_{qs} = \ast$, 
$\upsq t_{qs} = s'$, 
$\downsq t_{qs} = qs$;
\item 
$\leftsq t_{qs}^{a} = \ast$,
$\rightsq t_{qs}^{a} = qs$, 
$\upsq t_{qs}^{a} = q'a$, 
$\downsq t_{qs}^{a} = a$, where $a\ne\#$;
\item 
$\leftsq t_{qs}^{\#} = \otimes$,
$\rightsq t_{qs}^{\#} = qs$, 
$\upsq t_{qs}^{\#} = q'\#$, 
$\downsq t_{qs}^{\#} = \#$,
\end{itemize}
see Figure~\ref{fig:6}. 

\pagebreak[3]

\begin{figure}
\centering
\begin{tikzpicture}[scale=2.1]
\drawtileflattm{(0,0)}{(1,0)}{(1,1)}{(0,1)}{{$\ast$}}{{$a$}}{{$qs$}}{{$q'a$}}{$t_{qs}^{a}$}
\drawtileflattm{(1,0)}{(2,0)}{(2,1)}{(1,1)}{{$qs$}}{{$qs$}}{{$\ast$}}{{$s'$}}{$t_{qs}$}
\end{tikzpicture}
\hspace{1em}
\begin{tikzpicture}[scale=2.1]
\drawtileflattm{(0,0)}{(1,0)}{(1,1)}{(0,1)}{{$\otimes$}}{{$\#$}}{{$qs$}}{{$q'\#$}}{$t_{qs}^{\#}$}
\drawtileflattm{(1,0)}{(2,0)}{(2,1)}{(1,1)}{{$qs$}}{{$qs$}}{{$\ast$}}{{$s'$}}{$t_{qs}$}
\end{tikzpicture}
\caption{Tile types for the instruction $\delta_0\colon qs\mapsto q's'L$}
\label{fig:6}
\end{figure}
\Rem{%%%%%%%%%%%%%%%%%%%%%%%%%%
\begin{figure}
\centering
\begin{tikzpicture}[scale=2.1]
\drawtileflattm{(0,0)}{(1,0)}{(1,1)}{(0,1)}{${\ast}{\ast}$}{$\otimes$}{${\ast}{\ast}$}{$\Box$}{$t_\Box^{\ast\ast}$}
\end{tikzpicture}
\hspace{1em}
\begin{tikzpicture}[scale=2.1]
\drawtileflattm{(0,0)}{(1,0)}{(1,1)}{(0,1)}{$\numeral{k}$}{$\otimes$}{$\numeral{k\,{+}\,1}$}{$|$}{$t_k^\otimes$}
\end{tikzpicture}
\hspace{1em}
\begin{tikzpicture}[scale=2.1]
\drawtileflattm{(0,0)}{(1,0)}{(1,1)}{(0,1)}{$\numeral{k}$}{$\otimes$}{${\ast}{\ast}$}{$\Box$}{$t_k^{\ast\ast}$}
\end{tikzpicture}
\caption{Tile types $t_\Box^{\ast\ast}$, $t_k^\otimes$, and $t_k^{\ast\ast}$}
\label{fig:7}
\end{figure}%
Let $t_1 = t_{q_1\#}$; it is the same for all Turing machines. Finally, let $T_M$ be the set of all tile types constructed by~$M$.

Now, we can simulate the initial configuration of~$M$ with the blank tape. Indeed, to simulate the configuration $C = q_0\#\Box\Box\Box\ldots{}$, take the row of tiles whose types are
$$
t_0^{\phantom{i}}, 
t_\Box^{\ast\ast}, t_\Box^{\ast\ast}, t_\Box^{\ast\ast}, \ldots{},
$$
and then
$$
\begin{array}{lcl}
C & = & 
\upsq t_0^{\phantom{i}}\ 
\upsq t_\Box^{\ast\ast}\ \upsq t_\Box^{\ast\ast}\ \upsq t_\Box^{\ast\ast}\ \ldots{},
\end{array}
$$
see Figure~\ref{fig:8}. 

\begin{figure}
\centering
\begin{tikzpicture}[scale=2.1]
\drawtileflattm{(0,0)}{(1,0)}{(1,1)}{(0,1)}{${\otimes}$}{${\otimes}$}{$\numeral{0}$}{$q_0\#$}{$t_0$}
\drawtileflattm{(1,0)}{(2,0)}{(2,1)}{(1,1)}{${\numeral{0}}$}{${\otimes}$}{$\numeral{1}$}{$|$}{$t_0^\otimes$}
\drawtileflattm{(3-.2,0)}{(4-.2,0)}{(4-.2,1)}{(3-.2,1)}{${\numeral{k\,{-}\,1}}$}{${\otimes}$}{$\numeral{k}$}{$|$}{$t_{k-1}^\otimes$}
\drawtileflattm{(4-.2,0)}{(5-.2,0)}{(5-.2,1)}{(4-.2,1)}{${\numeral{k}}$}{${\otimes}$}{${\ast}{\ast}$}{$\Box$}{$t_{k}^{\ast\ast}$}
\drawtileflattm{(5-.2,0)}{(6-.2,0)}{(6-.2,1)}{(5-.2,1)}{${\ast}{\ast}$}{${\otimes}$}{${\ast}{\ast}$}{$\Box$}{$t_{\Box}^{\ast\ast}$}
\node [] at (2.42,0.5) {$\cdots$};
\node [] at (6.2,0.5) {$\cdots$};
\end{tikzpicture}
\caption{Simulation of the initial configuration with input $\numeral{k}$}
\label{fig:8}
\end{figure}

We shall use the new tile types together with the tile types of ${T}_{M_0}$ to simulate the computations of $M_0$ with the inputs we choose. To this end, let
$$
\begin{array}{lcl}
T_n & = & {T}_{M_0}\cup\{t_{k}^\otimes : k<n\}\cup\{t_{n}^{\ast\ast}\}\cup\{t_{\Box}^{\ast\ast}\}.
\end{array}
$$
}%%%%%%%%%%%%%%%%%%%%%%%%%

\begin{lemma}
\label{lem:TM}
There exists a unique $T_M$-tiling $f\colon\numN\times\numN\to T_M$ satisfying\/ \eqref{eq:T1},\/ \eqref{eq:T2}, and\/ \eqref{eq:T3}. Moreover, if $C_0,C_1,C_2,\ldots{}$ is the computation of $M$ on the blank tape, then
$$
\begin{array}{lcl}
C_k & = & \upsq f(0,k)\ \upsq f(1,k)\ \upsq f(2,k)\ \ldots,
\end{array}
$$
for every $k\in\numN$.
\end{lemma}

\pagebreak[2]

\begin{proof}
Induction on~$k$.
\end{proof}

\pagebreak[3]

%Due to Lemma~\ref{lem:TM}, a $T_M$-tiling $f_M\colon\numN\times\numN\to T_M$ such that $f_M(0,0)=t_0$ is defined uniquely.

%Now, we a ready to prove

To prove the next proposition, we have to define~$t_1$; let $t_1 = t_{q_1\#}$.

\begin{proposition}
Let $M$ be a Turing machine. Then, $M$ halts on the blank tape if, and only if, there exists a $T_M$-tiling satisfying \eqref{eq:T1}, \eqref{eq:T2}, \eqref{eq:T3}, and\/~\eqref{eq:T4}.
\end{proposition}

\begin{proof}
By Lemma~\ref{lem:TM}, there exists a unique $T_M$-tiling $f\colon\numN\times\numN\to T_M$ satisfying\/ \eqref{eq:T1},\/ \eqref{eq:T2}, and\/ \eqref{eq:T3}. Let us consider two possible cases for~$M$.

Assume that $M$ halts on the blank tape on some step~$k$ of the computation. Then we readily obtain that $f(0,k+j)=t_1$, for every $j\in\numN$. Thus, it satisfies \eqref{eq:T4}, as well.

Assume that $M$ does not halt on the blank tape. Then the $T_M$-tiling does not satisfy \eqref{eq:T4}, since $M$ never enters into the final configuration.
\end{proof}

\begin{corollary}
The tiling problem defined by \eqref{eq:T1}, \eqref{eq:T2}, \eqref{eq:T3}, and\/~\eqref{eq:T4}\/ is $\Sigma^0_1$-hard.
\end{corollary}
 
%%%%%%%%%%%%%%%%%%%%%%%%%%%%%%%%%%%%%%%%%%%%%

\pagebreak[3]

%\bibliographystyle{plain}
%\addcontentsline{toc}{section}{References} 
%\bibliography{sources}

\end{document}